\documentclass[12pt,reqno]{amsart}
\usepackage{graphics}
\usepackage{latexsym}
\usepackage[dvips]{epsfig}
\usepackage{color}
\usepackage{graphicx}
\usepackage{amsmath}
\usepackage{amssymb}
\usepackage{amsfonts}
\usepackage{eufrak}
\usepackage{amstext}
\usepackage{amsopn}
\usepackage{amsbsy}
\usepackage{amscd}
\usepackage{amsxtra}
\usepackage{amsthm}
\usepackage{bm}
\usepackage{CJK}

\newcommand{\R}{{\mathbb R}}

\newcommand{\beq}{\begin{equation}}
\newcommand{\eeq}{\end{equation}}
\newcommand{\ben}{\begin{eqnarray}}
\newcommand{\een}{\end{eqnarray}}
\newcommand{\beno}{\begin{eqnarray*}}
\newcommand{\eeno}{\end{eqnarray*}}

\newtheorem{thm}{Theorem}[section]
\newtheorem{defi}[thm]{Definition}
\newtheorem{lem}[thm]{Lemma}

\newtheorem{coro}[thm]{Corollary}
\newtheorem{rmk}[thm]{Remark}

\allowdisplaybreaks

\begin{document}

\title[Biharmonic Supercritical Problems]{A Monotonicity Formula and a Liouville-type Theorem for a Fourth Order Supercritical Problem}
\author[J. D\'avila]{Juan D\'avila}
\address{\noindent J. D\'avila -
Departamento de Ingenier\'{\i}a Matem\'atica and CMM, Universidad
de Chile, Casilla 170 Correo 3, Santiago, Chile.}
\email{jdavila@dim.uchile.cl}

\author[L. Dupaigne]{Louis Dupaigne}
\address{\noindent L. Dupaigne-
LAMFA, UMR CNRS 7352, Universite Picardie Jules Verne, 33 rue St. leu, 80039 Amiens, France}

\email{louis.dupaigne@math.cnrs.fr}

\author[K. Wang]{ Kelei Wang}
 \address{\noindent K. Wang-
 Wuhan Institute of Physics and Mathematics,
The Chinese Academy of Sciences, Wuhan 430071, China.}
\email{wangkelei@wipm.ac.cn}

\author[J. Wei]{Juncheng Wei}
\address{\noindent J. Wei -
Department Of Mathematics, Chinese University Of Hong Kong,
Shatin, Hong Kong and Department of Mathematics, University of British Columbia, Vancouver, B.C., Canada, V6T 1Z2. } \email{wei@math.cuhk.edu.hk}

\begin{abstract}
We consider Liouville-type and partial regularity results for the
nonlinear fourth-order problem
$$ \Delta^2 u=|u|^{p-1}u\ \ \mbox{in} \ \R^n,$$
where $ p>1$ and $n\ge1$. We give a complete classification of
stable and finite Morse index solutions (whether positive or sign
changing), in the full exponent range. We also compute an upper
bound of the Hausdorff dimension of the singular set of extremal
solutions. Our approach is motivated by Fleming's tangent cone
analysis technique for minimal surfaces and Federer's dimension
reduction principle in partial regularity theory. A key tool is the
monotonicity formula for biharmonic equations.
\end{abstract}



\maketitle

\date{}

\section{Introduction}
\setcounter{equation}{0}
We study the following model biharmonic superlinear elliptic equation
\begin{equation}\label{equation}
\Delta^2 u=|u|^{p-1}u \quad \mbox{in} \ \Omega,
\end{equation}
where $ \Omega \subset \R^n$ is a smoothly bounded domain or the
entire space and $p>1$ is a real number.  Inspired by the tangent
cone analysis in minimal surface theory, more precisely Fleming's
key observation that the existence of an entire nonplaner minimal
graph implies that of a singular
 area-minimizing cone (see his work on the Bernstein theorem
\cite{F}), we derive a monotonicity formula for solutions of
(\ref{equation}) to reduce the non-existence of nontrivial entire
solutions for the problem \eqref{equation}, to that of nontrivial
homogeneous solutions. Through this approach we give a complete
classification of stable solutions and those of finite Morse index,
whether positive or sign changing, when $\Omega=\R^{n}$ is  the
whole euclidean space. This in turn enables us to obtain partial
regularity as well as an estimate of the Hausdorff dimension of the
singular set of the extremal solutions in bounded domains.

\medskip

Let us first describe the monotonicity formula. Equation \eqref{equation} has two important features. It is variational, with energy functional given by
$$
\int \frac{1}{2}(\Delta u)^2-\frac{1}{p+1}|u|^{p+1}
$$
and it is invariant under the scaling transformation
$$
u^{\lambda}(x) = \lambda^{\frac4{p-1}}u(\lambda x).
$$
This suggests that the variations of the rescaled energy
$$
r^{4\frac{p+1}{p-1}-n}
\int_{B_r (x)}\frac{1}{2}(\Delta u)^2-\frac{1}{p+1}|u|^{p+1}
$$
with respect to the scaling parameter $r$ are meaningful. Augmented by the appropriate boundary terms, the above quantity is in fact nonincreasing. More precisely, take $u\in W_{loc}^{4,2}(\Omega) \cap L^{p+1}_{loc}(\Omega)$, fix $x \in \Omega$, let $0< r <R$ be such that $ B_r (x) \subset B_R (x) \subset \Omega$,  and define
\begin{eqnarray*}
 E(r;x,u)&:=&r^{4\frac{p+1}{p-1}-n}
\int_{B_r (x)}\frac{1}{2}(\Delta u)^2-\frac{1}{p+1}|u|^{p+1}\\
&&+\frac{2}{p-1}\left(n-2-\frac{4}{p-1}\right)r^{\frac{8}{p-1}+1-n}\int_{\partial
B_r (x)}u^2\\
&&+\frac{2}{p-1}\left(n-2-\frac{4}{p-1}\right)
\frac{d}{d r}\left(r^{\frac{8}{p-1}+2-n}\int_{\partial
B_r (x)}u^2\right)\\
&&+\frac{r^3}{2}\frac{d}{d r}\left[ r^{\frac{8}{p-1}+1-n}\int_{\partial
B_r (x)}\left(\frac{4}{p-1}r^{-1}u+\frac{\partial
u}{\partial r}\right)^2\right]\\
&&+\frac{1}{2}\frac{d}{d r}\left[r^{\frac{8}{p-1}+4-n}\int_{\partial
B_r (x)}\left(|\nabla u|^2-|\frac{\partial u}{\partial
r}|^2\right)\right]\\
&&+\frac{1}{2}r^{\frac{8}{p-1}+3-n}\int_{\partial
B_r (x)}\left(|\nabla u|^2-|\frac{\partial u}{\partial
r}|^2\right),
\end{eqnarray*}
where derivatives are taken in the sense of distributions.
Then, we have the following monotonicity formula.
\begin{thm}
\label{thm monotonicity}
Assume that
\begin{equation}
 \ n\geq 5, \qquad p>\frac{n+4}{n-4}.
\end{equation}
Let $u\in W_{loc}^{4,2}(\Omega) \cap L^{p+1}_{loc}(\Omega)$ be a weak solution of \eqref{equation}.
Then, $E(r;x,u)$ is non-decreasing in $r \in (0, R)$. Furthermore there is a constant $c(n,p)>0$
such that
\begin{equation}
\frac{d}{dr}E(r;0,u)\geq c(n,p)r^{-n+2+\frac{8}{p-1}}\int_{\partial
B_r}\left(\frac{4}{p-1}r^{-1}u+\frac{\partial u}{\partial
r}\right)^2.
\end{equation}
\end{thm}

\begin{rmk}
Monotonicity formulae have a long history that we will not describe here. Let us simply mention two earlier results that seem closest to our findings: the formula of Pacard \cite{pacard-mm} for the classical Lane-Emden equation and the one of Chang, Wang and Yang \cite{CWY} for biharmonic maps.
\end{rmk}

\medskip

Consider again equation (\ref{equation}) in the case where  $\Omega=\R^n$, i.e.,
\begin{equation}
\label{1.1} \Delta^2 u = |u|^{p-1} u \quad \mbox{in $\R^n$}.
\end{equation}
Let
$$
p_{S}(n)=\left\{
\begin{aligned}
+\infty&\quad\text{if $n\le 4$}\\
\frac{n+4}{n-4}&\quad\text{if $n\ge 5$}
\end{aligned}
\right.
$$
denote the Sobolev exponent. When $1<p\leq p_{S}(n)$, all positive
solutions to (\ref{1.1}) are classified: if $p<p_{S}(n)$,
then $ u \equiv 0$; if $p=p_{S}(n)$, then all solutions can
be written in the form $ u= c_n (\frac{\lambda}{\lambda^2
+|x-x_0|^2})^{\frac{n-4}{2}}$ for some $c_n>0, \lambda>0, x_0 \in
\R^n$, see the work of Xu and one of the authors \cite{WX}. However, there can be many sign-changing solutions to the equation (see the work by Guo, Li and one of the authors \cite{GLW} for the critical case $p=p_{S}(n)$).

%Thus from now on we restrict to the supercritical case
%$p>\frac{n+4}{n-4}$.
Here, we allow $u$ to be sign-changing and $p$ to be supercritical. Instead, we restrict the analysis to solutions having finite Morse index. More precisely, define the quadratic form
\begin{equation}
\label{1.2} \Lambda_u(\phi) := \int_{\R^n} |\Delta \phi|^2 dx -p
\int_{\R^n} |u|^{p-1} \phi^2 dx, \quad \forall\; \phi \in H^2
(\R^n).
\end{equation}
A solution to \eqref{1.1} is said to be stable if
\begin{equation*}
%\label{stable}
\int_{\R^n} |\Delta \phi|^2 dx -p \int_{\R^n} |u|^{p-1} \phi^2 dx
\geq 0, \quad \forall\; \phi \in H^2 (\R^n).
\end{equation*}
More generally, the Morse index of a solution is defined as the maximal dimension of all
subspaces $E$ of $H^2 (\R^n)$ such that $\Lambda_u(\phi) < 0$ in $E\setminus\{0\}$. Clearly, a solution is stable if and only if its Morse index is equal to zero. It is also standard knowledge that
if a solution to (\ref{1.1}) has finite Morse index, then there is a compact set ${\mathcal K} \subset \R^n$ such that
\begin{equation*}
%\label{stable}
\int_{\R^n} |\Delta \phi|^2 dx -p \int_{\R^n} |u|^{p-1} \phi^2 dx
\geq 0, \quad \forall\; \phi \in H^2 (\R^n \backslash {\mathcal K}).
\end{equation*}
%Stable solutions exist  when $p$ is large. More precisely, letting
Recall that if
\begin{equation}\label{singsing}
\gamma=\frac{4}{p-1},\qquad K_0 = \gamma (\gamma+2)( \gamma -n+4)(\gamma -n+2),
\end{equation}
%it is easy to see that
then
\begin{equation}
\label{sing}
 u_s (r) = K_0^{1/(p-1)} r^{-4/(p-1)}
\end{equation}
is a singular solution to (\ref{1.1}) in $ \R^n \setminus \{0\}$. By the Hardy-Rellich inequality with best constant \cite{Re}
\begin{equation*}
%\label{Hardy inequality}
\int_{\R^n} |\Delta \phi|^2 dx \geq \frac{n^2 (n-4)^2}{16}
\int_{\R^n} \frac{\phi^2}{|x|^4} dx, \;\;\; \forall \phi \in H^2
(\R^n),
\end{equation*}
the singular solution $u_s$ is stable if and only if
\begin{equation} \label{cond}
p K_0 \le \frac{n^2 (n-4)^2}{16}.
\end{equation}
Solving the corresponding quartic equation, \eqref{cond} holds if and only if $p\ge p_{c}(n)$ where $p_{c}(n)>p_{S}(n)$ is the fourth-order Joseph-Lundgren exponent computed by Gazzola and Grunau \cite{GG}:
$$
p_{c}(n)=\left\{
\begin{aligned}
+\infty&\quad\text{if $n\le 12$}\\
\frac{n+2-\sqrt{n^2+4-n\sqrt{n^2-8n+32} } }{n-6-\sqrt{n^2+4-n\sqrt{n^2-8n+32} }}&\quad\text{if $n\ge 13$}
\end{aligned}
\right.
$$
Equivalently, for fixed $p>p_{S}(n)$, define $n_p$ to be the smallest dimension such that \eqref{cond} holds. Then,
$$
\eqref{cond}\Longleftrightarrow p \geq p_c (n) \Longleftrightarrow n\ge n_{p}.
$$

%\begin{equation}
%p K_0<\frac{ n^2 (n-4)^2}{16} \ \mbox{if and only if} \
%\end{equation}
The existence, uniqueness and stability of regular radial positive solutions to (\ref{1.1}) is by now well understood (see  the works of Gazzola-Grunau, of Guo and one of the authors, and of Karageorgis \cite{GG, GW2, Ka}): for each $a>0$ there exists a unique
entire radial positive solution $ u_a (|x|)$ to \eqref{1.1} with $u_a(0)=a$. This radial
positive solution is stable if and only if \eqref{cond} holds.

In our second result, which is a Liouville-type theorem,  we give a
complete characterization of all finite Morse index solutions (whether radial or not, whether positive or not).

\medskip

\begin{thm}
\label{thmstable}
Let $u$ be  solution to (\ref{1.1}) with finite
Morse index.
\begin{itemize}
\item If $p \in (1, p_c(n))$, $p\neq p_{S}(n)$, then $ u \equiv 0 $;
\item If $p=p_{S}(n)$, then $u$ has finite energy i.e.
\[\int_{\mathbb{R}^n}(\Delta u)^2=\int_{\mathbb R^{n}}|u|^{p+1}<+\infty.\]
If in addition $u$ is stable, then in fact $u\equiv 0$.
\end{itemize}
\end{thm}

\begin{rmk}
{\it  According to the preceding discussions, Theorem
\ref{thmstable} is sharp: on the one hand, in the critical case $p=p_{S}(n)$,
Guo, Li and one of the authors \cite{GLW} have constructed a large class of
solutions to \eqref{equation} with finite energy. Since in this case
$\frac{(p-1)n}{4}=p+1$, by a result of Rozenbljum \cite{Ro}, such solutions have
finite Morse index. On the other hand, for $p\ge p_{c}(n)$, all radial solutions are stable (see \cite{GW2, Ka}). }
\end{rmk}
\begin{rmk}
{\it The above theorem generalizes a similar result of Farina
\cite{Fa} for the classical Lane-Emden equation.
 }
\end{rmk}
Now consider \eqref{equation}  when $\Omega$ is a smoothly bounded domain of $\R^{n}$ and supplement it with Navier boundary conditions:
\begin{equation}\label{boundary value problem}
\left\{
\begin{array}{ll}
\Delta^2 u = \lambda(u + 1)^p & \mbox{ in } \Omega\\
u = \Delta u = 0 & \mbox{ on } \partial \Omega,
\end{array}
\right.
\end{equation}
where $\lambda>0$ is a parameter.
It is well known that there exists a
critical value $\lambda^* > 0$ depending on $p$ and $\Omega$
such that
\begin{itemize}
\item If $\lambda \in (0, \lambda^*)$, \eqref{boundary value problem} has a minimal and classical solution $u_\lambda$, which is positive and stable;
\item If $\lambda = \lambda^*$, a unique weak solution, called the extremal solution $u_{\lambda^\ast}$ exists for $(P_{\lambda^*})$. It is given as the pointwise limit $u_{\lambda^\ast} = \lim_{\lambda\uparrow} u_\lambda$
;
\item No weak solution of \eqref{boundary value problem} exists whenever $\lambda > \lambda^*$.
\end{itemize}

%{\cb
%For these facts we refer to ...
%The extremal solution $u_{\lambda^\ast}$ satisfies \eqref{boundary value problem} in the sense that $u_{\lambda^\ast} \in W^{2,2}(\Omega) \cap L^p(\Omega)$, $u=0$ on $\partial \Omega)$ and
%$$
%\int_\Omega \Delta u_{\lambda^\ast} \Delta \varphi dx
%= \lambda^\ast \int_\Omega u_{\lambda^\ast} \varphi dx
%$$
%for all $\varphi \in C^2(\overline \Omega)$ with $\varphi = 0$ on $\partial \Omega$.}

%It is standard to show that $u_{\lambda^\ast}$ is a {\em weak solution} in a suitable sense that we shall not define here.
An outstanding  remaining problem  is  the regularity of the extremal solution $ u_{\lambda^\ast}$.
%Fixing $p>\frac{n+4}{n-4}$, define $n_p$ to be the smallest
%dimension such that $ p \geq p_c (n)$. This is also the smallest
%dimension such that the Liouville theorem for finite Morse index solutions, Theorem \ref{thmstable}, does not hold.
An application of Theorem \ref{thmstable} and standard blow-up analysis gives
\begin{thm}
\label{thm full regularity}
If $n<n_p$ (equivalently $p<p_{c}(n)$), the extremal solution $u_{\lambda^\ast}$ is smooth.
\end{thm}
More generally,
\begin{thm}
\label{thm full regularity2}
Assume $n<n_p$ (equivalently $p<p_{c}(n)$).
\begin{itemize}
\item  Let $\Omega$ be a smoothly bounded domain and $u\in H^1_0(\Omega)\cap H^{2}(\Omega)$ be a solution of \eqref{boundary value problem} of finite Morse index $k\in\mathbb N$. Then, $u$ is smooth and there exists a constant $C>0$ depending only on $k,N,\Omega,p$  such that
$$
\| u \|_{L^\infty(\Omega)} \le C.
$$
In particular, any classical solution satisfies the above inequality.
\item Let $\Omega$ be any open set and $u\in H^1_0(\Omega)\cap H^{2}(\Omega)$ be a solution of \eqref{equation}. Then, there exists a constant $C>0$ depending only on $k,N,\Omega,p$  such that for every $i\le3$,
$$
\vert \nabla^{i}u \vert\le C\text{dist }(x,\partial\Omega)^{-\frac{4}{p-1}-i} \quad\text{a.e.in $\Omega$}
$$

\end{itemize}
\end{thm}
The first part of the above theorem is again sharp since the singular solution defined by \eqref{sing}, \eqref{singsing} is stable whenever $n\ge n_{p}$.  For such dimensions, one can still try to estimate the size of the singular set of solutions.
\begin{defi}
%\label{defintion}
A point $x$ belongs to the regular set of a function $u\in L^{1}_{loc}(\Omega)$ if there exists a neighborhood $B$ of $x$ such that $u\in L^{\infty}(B)$. Otherwise, $x$ belongs to ${\mathcal S}$, the singular set of $u$.
\end{defi}
By definition, the regular set is an open set. By elliptic estimates
applied to \eqref{equation}, $u$ is smooth in its regular set. Now,
we state the interior partial regularity for $u_{\lambda^\ast}$.

\begin{thm}\label{thm partial regularity}
Let $ n \geq n_p$ and let $u_{\lambda^\ast}$ be the extremal solution to \eqref{boundary
value problem}. Then the Hausdorff dimension of its singular set ${\mathcal S}$ is no more than
$n-n_p$. Moreover, when $n=n_p$, ${\mathcal S}$ is a discrete set.
\end{thm}
We now list some known results. We start with the analogous second order equation
\begin{equation}
\label{second} \Delta u+ |u|^{p-1} u=0, \ \mbox{in} \ \R^n.
\end{equation}
As mentioned earlier, Farina completely classified finite Morse index solutions (positive or sign-changing) in his seminal paper \cite{Fa}.
%The main result of \cite{Fa} is that no finite Morse solution exists to (\ref{second}) if $p <p_{c}(n)$, $p\neq p_{S}(n)$. Here $p_{S}(n), p_{c}(n)$ denote the well-known standard Sobolev and Joseph-Lundgren exponent (\cite{GNW}). In addition, stable radial solution exist for $p \geq p_{JL}$ (\cite{GNW}).
His proof makes a delicate use of the classical Moser iteration method. More precisely, if one multiplies the equation (\ref{second}) by a power of $u$, say $ u^q$, $q>1$, Moser's iteration works because of the following simple identity
\begin{equation*}
\int_{\mathbb{R}^n} u^q (-\Delta u)= \frac{4q}{(q+1)^2} \int_{\mathbb{R}^n} |\nabla u^{\frac{q+1}{2}}|^2
, \quad \forall u \in C_0^1 (\mathbb{R}^n).
\end{equation*}

There have been many attempts
%{\em ``incomplete attempts''}
to generalize Moser's iteration technique (or Farina's approach) to fourth order problems like (\ref{equation}). Unfortunately, this runs into problems: the corresponding identity reads
\begin{equation*}
\int_{\mathbb{R}^n} u^q (\Delta^2u)= \frac{4q}{(q+1)^2}  \int_{\mathbb{R}^n} |\Delta u^{\frac{q+1}{2}}|^2 -q (q-1)^2 \int_{\mathbb{R}^n} u^{q-3} |\nabla u|^4, \forall u \in C_0^2 (\mathbb{R}^n)
\end{equation*}
and the additional term $\int_{\mathbb{R}^n} u^{q-3} |\nabla u|^4$
makes the Moser iteration argument difficult to use.

Another strategy is to use the test function $v=-\Delta u$. This allows to treat exponents less than $ \frac{n}{n-8}+\epsilon_n$ for some $ \epsilon_n>0$, see the works of Cowan-Ghoussoub-Esposito \cite{CEG} and Ye and one of the authors \cite{Y-W}. Another approach, obtained by Cowan and Ghoussoub\footnote{a similar method was first announced in \cite{dupaigne}, and later published in the work by Farina-Sirakov and one of the authors \cite{DFS}.}\cite{CG}, and further exploited by Hajlaoui, Harrabi and Ye \cite{HHY}, is to derive the following interesting interpolated version of the inequality: for stable solutions to (\ref{equation}), there holds
\begin{equation*}
%\label{secondstable}
\sqrt{p} \int_{\R^n} |u|^{\frac{p-1}{2}} \phi^2 \leq \int_{\R^n}
|\nabla \phi |^2, \quad \forall \phi \in C_0^1 (\R^n).
\end{equation*}
This approach improves the first upper bound $\frac{n}{n-8}+\epsilon_{n}$, but it again fails to catch the optimal exponent $p_c(n)$ (when $n\geq 13$). It should be remarked that by combining these two approaches one can show that stable positive solutions to (\ref{equation}) do not exist when $n \leq 12$ and $ p>\frac{n+4}{n-4}$, see \cite{HHY}.

In the above references, only positive solutions to (\ref{equation}) are considered. One reason is their use of the following inequality, due to Souplet (\cite{Soup})
\begin{equation}
\label{souplet}
\Delta u + \frac{2}{p+1} u^{\frac{p+1}{2}} \leq 0 \ \ \mbox{in} \ \R^n.
\end{equation}
As observed in \cite{DGGW} for a similar equation, the use of the
above inequality can be completely avoided.

In this paper we take a completely new approach, which also avoids
the use of \eqref{souplet} and requires minimal integrability. One
of our motivations is Fleming's proof of the Bernstein theorem for
minimal surfaces in dimension $3$. Fleming used a monotonicity
formula for minimal surfaces together with a compactness result to
blow down the minimal surface. It turns out that the blow-down limit
is a minimal cone. This is because the monotonic quantity is
constant only for minimizing cones. Then, he proved that minimizing
cones are flat, which implies in turn the flatness of the original
minimal surface.

At last, let us sketch the proof of Theorem \ref{thmstable}: we
first derive a monotonicity formula for our equation
\eqref{equation}. Then, we classify stable solutions: this is
Theorem \ref{thm Liouville for stable} in Section \ref{sect
liouville stable}. To do this, we estimate solutions in the
$L^{p+1}$ norm, utilizing the afore-mentioned methods available in
the litterature, and then show that the {\em blow-down} limit
$u^\infty(x) = \lim_{\lambda\to\infty} \lambda^{\frac{4}{p-1}}
u(\lambda x)$ satisfies $E(r) \equiv const$. Then, Theorem \ref{thm
monotonicity} implies that $u^\infty$ is a homogeneous stable
solution, and we show in Theorem~\ref{thm homogeneous solution} that
such solutions are trivial if $p<p_c(n)$. Then similar to Fleming's
proof, the triviality of  the {\em blow-down} limit implies that the
original entire solution is also trivial. Finally, we extend our
result to solutions of finite Morse index.

\bigskip

\noindent
{\bf Acknowledgment.}  L. Dupaigne thanks J. Wei and the math department of the Chinese University of Hong Kong (where part of this work was done) for their warm hospitality. Kelei Wang is partially supported by the Joint Laboratory of CAS-Croucher in Nonlinear PDE. Juncheng Wei was supported by a GRF grant
from RGC of Hong Kong.
J. D\'avila acknowledges support of Fondecyt  1090167, CAPDE-Anillo ACT-125 and
Fondo Basal CMM.

\section{Proof of the Monotonicity formula}
\setcounter{equation}{0}

In this section we derive a monotonicity formula for functions $u\in W^{4,2}(B_R(0))\cap L^{p+1}(B_R(0))$
 solving \eqref{equation} in $B_R(0) \subset \Omega$. We assume that $p>\frac{n+4}{n-4}$.

\begin{proof}[Proof of Theorem \ref{thm monotonicity}]
Since the boundary integrals in $E(r;x,u)$ only involve second order
derivatives of $u$, the boundary integrals in $\frac{dE}{d
r}(r;x,u)$ only involve third order derivatives of $u$. By our
assumption $u\in W^{4,2}(B_R(0))\cap L^{p+1}(B_R(0))$, for each $B_r
(x)\subset B_R(0)$, $u\in W^{3,2}(\partial B_r (x))$. Thus, the
following calculations can be rigorously verified. Assume that $x=0$
and that the balls $B_\lambda$ are all centered at $0$. Take
\[\widetilde{E}(\lambda):=\lambda^{4\frac{p+1}{p-1}-n}
\int_{B_\lambda}\frac{1}{2}(\Delta u)^2-\frac{1}{p+1}|u|^{p+1}.\]
Define
$$
v:= \Delta u
$$
and
\[u^\lambda(x):=\lambda^{\frac{4}{p-1}}u(\lambda x),
\ \ \ \ v^\lambda(x):=\lambda^{\frac{4}{p-1}+2}v(\lambda x).\] We
still have $v^\lambda=\Delta u^\lambda$, $\Delta
v^\lambda=(u^\lambda)^p$, and by differentiating in $\lambda$,
\[\Delta\frac{du^\lambda}{d\lambda}=\frac{dv^\lambda}{d\lambda}.\]
Note that differentiation in $\lambda$ commutes with differentiation
and integration in $x$.
A rescaling shows
\[\widetilde{E}(\lambda)=\int_{B_1}\frac{1}{2}(v^\lambda)^2-\frac{1}{p+1}|u^\lambda|^{p+1}.\]
Hence
\begin{eqnarray}\label{2.1}
\frac{d}{d\lambda}\widetilde{E}(\lambda)&=&\int_{B_1}v^\lambda\frac{dv^\lambda}{d\lambda}
-(u^\lambda)^{p}\frac{du^\lambda}{d\lambda}\\\nonumber
&=&\int_{B_1}v^\lambda\Delta \frac{du^\lambda}{d\lambda} -\Delta
v^\lambda\frac{du^\lambda}{d\lambda}\\\nonumber
 &=&\int_{\partial
B_1}v^\lambda\frac{\partial}{\partial r} \frac{du^\lambda}{d\lambda}
-\frac{\partial v^\lambda}{\partial r}\frac{du^\lambda}{d\lambda}.
\end{eqnarray}
In what follows, we express all derivatives of $u^{\lambda}$ in the
$r=|x|$ variable in terms of derivatives in the $\lambda$ variable.
In the definition of $u^\lambda$ and $v^\lambda$, directly
differentiating in $\lambda$ gives
\begin{equation}\label{2.3}
\frac{du^\lambda}{d\lambda}(x)=\frac{1}{\lambda}
\left(\frac{4}{p-1}u^\lambda(x)+r\frac{\partial u^\lambda}{\partial
r}(x)\right),
\end{equation}
\begin{equation}\label{2.4}
\frac{dv^\lambda}{d\lambda}(x)=\frac{1}{\lambda}
\left(\frac{2(p+1)}{p-1}v^\lambda(x)+r\frac{\partial
v^\lambda}{\partial r}(x)\right).
\end{equation}
In \eqref{2.3}, taking derivatives in $\lambda$ once again, we get
\begin{equation}\label{2.5}
\lambda\frac{d^2u^\lambda}{d\lambda^2}(x)+\frac{du^\lambda}{d\lambda}(x)
=\frac{4}{p-1}\frac{du^\lambda}{d\lambda}(x) +r\frac{\partial
}{\partial r}\frac{du^\lambda}{d\lambda}(x).
\end{equation}
 Substituting \eqref{2.4} and \eqref{2.5} into \eqref{2.1} we obtain
\begin{eqnarray}
\frac{d \widetilde{E}}{d\lambda}&=&\int_{\partial
B_1}v^\lambda\left(\lambda\frac{d^2u^\lambda}{d\lambda^2}+\frac{p-5}{p-1}\frac{du^\lambda}{d\lambda}\right)
-\frac{du^\lambda}{d\lambda}\left(\lambda\frac{dv^\lambda}{d\lambda}-\frac{2(p+1)}{p-1}v^\lambda\right)\\
\nonumber &=&\int_{\partial B_1}\lambda
v^\lambda\frac{d^2u^\lambda}{d\lambda^2}+3v^\lambda\frac{du^\lambda}{d\lambda}
-\lambda\frac{du^\lambda}{d\lambda}\frac{dv^\lambda}{d\lambda}.\label{2.2}
\end{eqnarray}
Observe that $v^{\lambda}$ is expressed as a combination of $x$ derivatives of $u^{\lambda}$. So we also transform $v^\lambda$ into $\lambda$ derivatives of $u^\lambda$. By taking derivatives in $r$ in \eqref{2.3} and noting
\eqref{2.5}, we get on $\partial B_1$,
\begin{eqnarray*}
\frac{\partial^2u^\lambda}{\partial
r^2}&=&\lambda\frac{\partial}{\partial
r}\frac{du^\lambda}{d\lambda}-\frac{p+3}{p-1}\frac{\partial
u^\lambda}{\partial r}\\
&=&\lambda^2\frac{d^2u^\lambda}{d\lambda^2}+\frac{p-5}{p-1}\lambda\frac{du^\lambda}{d\lambda}
-\frac{p+3}{p-1}\left(\lambda\frac{du^\lambda}{d\lambda}-\frac{4}{p-1}u^\lambda\right)\\
&=&\lambda^2\frac{d^2u^\lambda}{d\lambda^2}-\frac{8}{p-1}\lambda\frac{du^\lambda}{d\lambda}
+\frac{4(p+3)}{(p-1)^2}u^\lambda.
\end{eqnarray*}
Then on $\partial B_1$,
\begin{eqnarray*}
v^\lambda&=&\frac{\partial^2u^\lambda}{\partial r^2}
+\frac{n-1}{r}\frac{\partial u^\lambda}{\partial
r}+\frac{1}{r^2}\Delta_\theta u^\lambda\\
&=&\lambda^2\frac{d^2u^\lambda}{d\lambda^2}-\frac{8}{p-1}\lambda\frac{du^\lambda}{d\lambda}
+\frac{4(p+3)}{(p-1)^2}u^\lambda+(n-1)\left(\lambda\frac{du^\lambda}{d\lambda}-\frac{4}{p-1}u^\lambda\right)
+\Delta_\theta u^\lambda\\
&=&\lambda^2\frac{d^2u^\lambda}{d\lambda^2}+\left(n-1-\frac{8}{p-1}\right)\lambda\frac{du^\lambda}{d\lambda}
+\frac{4}{p-1}(\frac{4}{p-1}-n+2)u^\lambda+\Delta_\theta u^\lambda.
\end{eqnarray*}
Here $\Delta_\theta$ is the Beltrami-Laplace operator on $\partial
B_1$ and below $\nabla_\theta$ represents the tangential
derivative on $\partial B_1$. For notational convenience, we also define the
constants
\[\alpha=n-1-\frac{8}{p-1},\quad\quad \beta=\frac{4}{p-1}(\frac{4}{p-1}-n+2).\]
Now \eqref{2.2} reads
\begin{eqnarray*}
\frac{d}{d\lambda}\widetilde{E}(\lambda)&=&\int_{\partial
B_1}\lambda
\left(\lambda^2\frac{d^2u^\lambda}{d\lambda^2}+\alpha\lambda\frac{du^\lambda}{d\lambda}
+\beta
u^\lambda\right)\frac{d^2u^\lambda}{d\lambda^2}\\
&&+3\left(\lambda^2\frac{d^2u^\lambda}{d\lambda^2}+\alpha\lambda\frac{du^\lambda}{d\lambda}
+\beta u^\lambda\right)\frac{du^\lambda}{d\lambda}\\
&&-\lambda\frac{du^\lambda}{d\lambda}\frac{d}{d\lambda}\left(\lambda^2\frac{d^2u^\lambda}{d\lambda^2}+\alpha\lambda\frac{du^\lambda}{d\lambda}
+\beta u^\lambda\right)\\
&&+\int_{\partial B_1}\lambda\Delta_\theta
u^\lambda\frac{d^2u^\lambda}{d\lambda^2}+3\Delta_\theta
u^\lambda\frac{du^\lambda}{d\lambda}
-\lambda\frac{du^\lambda}{d\lambda}\Delta_\theta\frac{du^\lambda}{d\lambda}\\
&=&R_1+R_2.
\end{eqnarray*}
Integrating by parts on $\partial B_1$, we get
\begin{eqnarray*}
R_2&=&\int_{\partial B_1}-\lambda\nabla_\theta
u^\lambda\nabla_\theta\frac{d^2u^\lambda}{d\lambda^2}-3\nabla_\theta
u^\lambda\nabla_\theta\frac{du^\lambda}{d\lambda}
+\lambda\big|\nabla_\theta\frac{du^\lambda}{d\lambda}\big|^2\\
&=&-\frac{\lambda}{2}\frac{d^2}{d\lambda^2}\left(\int_{\partial
B_1}|\nabla_\theta
u^\lambda|^2\right)-\frac{3}{2}\frac{d}{d\lambda}\left(\int_{\partial
B_1}|\nabla_\theta u^\lambda|^2\right)+2\lambda\int_{\partial
B_1}|\nabla_\theta\frac{du^\lambda}{d\lambda}|^2\\
&=&-\frac{1}{2}\frac{d^2}{d\lambda^2}\left(\lambda\int_{\partial
B_1}|\nabla_\theta
u^\lambda|^2\right)-\frac{1}{2}\frac{d}{d\lambda}\left(\int_{\partial
B_1}|\nabla_\theta u^\lambda|^2\right)+2\lambda\int_{\partial
B_1}|\nabla_\theta\frac{du^\lambda}{d\lambda}|^2\\
&\geq&-\frac{1}{2}\frac{d^2}{d\lambda^2}\left(\lambda\int_{\partial
B_1}|\nabla_\theta
u^\lambda|^2\right)-\frac{1}{2}\frac{d}{d\lambda}\left(\int_{\partial
B_1}|\nabla_\theta u^\lambda|^2\right).
\end{eqnarray*}
For $R_1$, after some simplifications we obtain
\begin{eqnarray*}
R_1&=&\int_{\partial B_1}\lambda
\left(\lambda^2\frac{d^2u^\lambda}{d\lambda^2}+\alpha\lambda\frac{du^\lambda}{d\lambda}
+\beta u^\lambda\right)\frac{d^2u^\lambda}{d\lambda^2}\\
&&\ \ \ \ \ \
+3\left(\lambda^2\frac{d^2u^\lambda}{d\lambda^2}+\alpha\lambda\frac{du^\lambda}{d\lambda}
+\beta u^\lambda\right)\frac{du^\lambda}{d\lambda}\\
&&\ \ \ \ \ \ -\lambda\frac{du^\lambda}{d\lambda}\left(
\lambda^2\frac{d^3u^\lambda}{d\lambda^3}+(2+\alpha)\lambda\frac{d^2u^\lambda}{d\lambda^2}
+(\alpha+\beta)\frac{du^\lambda}{d\lambda}\right)\\
&=&\int_{\partial
B_1}\lambda^3\left(\frac{d^2u^\lambda}{d\lambda^2}\right)^2
+\lambda^2\frac{d^2u^\lambda}{d\lambda^2}\frac{du^\lambda}{d\lambda}
+\beta \lambda u^\lambda\frac{d^2u^\lambda}{d\lambda^2}+3\beta
u^\lambda\frac{du^\lambda}{d\lambda}\\
&&\ \ \ \ \
+(2\alpha-\beta)\lambda\left(\frac{du^\lambda}{d\lambda}\right)^2
-\lambda^3\frac{du^\lambda}{d\lambda}\frac{d^3u^\lambda}{d\lambda^3}\\
&=&\int_{\partial
B_1}2\lambda^3\left(\frac{d^2u^\lambda}{d\lambda^2}\right)^2
+4\lambda^2\frac{d^2u^\lambda}{d\lambda^2}\frac{du^\lambda}{d\lambda}
+(2\alpha-2\beta)\lambda\left(\frac{du^\lambda}{d\lambda}\right)^2 \\
&&\ \ \ \ \
+\frac{\beta}{2}\frac{d^2}{d\lambda^2}\left[\lambda\left(
u^\lambda\right)^2\right]-
\frac{1}{2}\frac{d}{d\lambda}\left[\lambda^3\frac{d}{d\lambda}\left(\frac{du^\lambda}{d\lambda}\right)^2\right]
+\frac{\beta}{2} \frac{d}{d\lambda}\left(u^\lambda\right)^2.
\end{eqnarray*}
Here we have used the relations
\[\lambda ff^{\prime\prime}=\left(\frac{\lambda}{2}f^2\right)^{\prime\prime}-2ff^{\prime}-\lambda(f^{\prime})^2,\]
and
\[-\lambda^3f^{\prime}f^{\prime\prime\prime}=-\left[\frac{\lambda^3}{2}\left((f^{\prime})^2\right)^{\prime}\right]^{\prime}+3\lambda^2
f^{\prime}f^{\prime\prime}+\lambda^3(f^{\prime\prime})^2.\] Since
$p>\frac{n+4}{n-4}$, direct calculations show that
\begin{equation}
\label{alphabeta}
\alpha-\beta=\left(n-1-\frac{8}{p-1}\right)-\frac{4}{p-1}\left(\frac{4}{p-1}-n+2\right)>1.
\end{equation}
Thus,
\begin{eqnarray}\label{2.6}
&&2\lambda^3\left(\frac{d^2u^\lambda}{d\lambda^2}\right)^2
+4\lambda^2\frac{d^2u^\lambda}{d\lambda^2}\frac{du^\lambda}{d\lambda}
+(2\alpha-2\beta)\lambda\left(\frac{du^\lambda}{d\lambda}\right)^2\\\nonumber
&=&2\lambda\left(\lambda\frac{d^2u^\lambda}{d\lambda^2}+\frac{du^\lambda}{d\lambda}\right)^
2+(2\alpha-2\beta-2)\lambda\left(\frac{du^\lambda}{d\lambda}\right)^2\\\nonumber
&\geq&0.
\end{eqnarray}
Then,
\begin{eqnarray*}
R_1 &\geq&\int_{\partial
B_1}\frac{\beta}{2}\frac{d^2}{d\lambda^2}\left[\lambda\left(
u^\lambda\right)^2\right]-
\frac{1}{2}\frac{d}{d\lambda}\left[\lambda^3\frac{d}{d\lambda}\left(\frac{du^\lambda}{d\lambda}\right)^2\right]
+\frac{\beta}{2} \frac{d}{d\lambda}\left(u^\lambda\right)^2.
\end{eqnarray*}
Now, rescaling back, we can write those  $\lambda$ derivatives in
$R_1$ and $R_2$ as follows.
\[\int_{\partial B_1}\frac{d}{d\lambda}\left(u^\lambda\right)^2
=\frac{d}{d\lambda}\left(\lambda^{\frac{8}{p-1}+1-n}\int_{\partial
B_\lambda}u^2\right).\]
\[\int_{\partial B_1}\frac{d^2}{d\lambda^2}\left[\lambda\left(
u^\lambda\right)^2\right]
=\frac{d^2}{d\lambda^2}\left(\lambda^{\frac{8}{p-1}+2-n}\int_{\partial
B_\lambda}u^2\right).\]
\[\int_{\partial B_1}\frac{d}{d\lambda}\left[\lambda^3\frac{d}{d\lambda}\left(\frac{du^\lambda}{d\lambda}\right)^2\right]
=\frac{d}{d\lambda}\left[\lambda^3\frac{d}{d\lambda}\left(\lambda^{\frac{8}{p-1}+1-n}\int_{\partial
B_\lambda}\left(\frac{4}{p-1}\lambda^{-1}u+\frac{\partial
u}{\partial r}\right)^2\right)\right].\]
\[\frac{d^2}{d\lambda^2}\left(\lambda\int_{\partial
B_1}|\nabla_\theta
u^\lambda|^2\right)=\frac{d^2}{d\lambda^2}\left[\lambda^{1+\frac{8}{p-1}+2+1-n}\int_{\partial
B_\lambda}\left(|\nabla u|^2-|\frac{\partial u}{\partial
r}|^2\right)\right].\]
\[\frac{d}{d\lambda}\left(\int_{\partial
B_1}|\nabla_\theta
u^\lambda|^2\right)=\frac{d}{d\lambda}\left[\lambda^{\frac{8}{p-1}+2+1-n}\int_{\partial
B_\lambda}\left(|\nabla u|^2-|\frac{\partial u}{\partial
r}|^2\right)\right]\]
 Substituting these into
$\frac{d}{d\lambda}E(\lambda;0,u)$ we finish the proof.
\end{proof}

\bigskip

\noindent Denote $c(n,p)=2\alpha-2\beta-2>0$. By \eqref{2.6}, we have
\begin{coro}\label{coro characterization of homogeneous solutions}
\[\frac{d}{dr}E(r;0,u)\geq c(n,p)r^{-n+2+\frac{8}{p-1}}\int_{\partial
B_r}\left(\frac{4}{p-1}r^{-1}u+\frac{\partial u}{\partial
r}\right)^2.\]
In particular, if $E(\lambda;0,u)\equiv const.$ for
all $\lambda\in(r,R)$,  $u$ is homogeneous in $B_R\setminus B_r$:
\[u(x)=|x|^{-\frac{4}{p-1}}u\left(\frac{x}{|x|}\right).\]
\end{coro}
We  end this section with the following observation : in the above computations we just need the inequality (\ref{alphabeta}) to hold. In particular the formula can be easily extended to biharmonic equations with negative exponents. We state the following monotonicity formula for solutions of
\begin{equation}
\label{negative}
\Delta^2 u = -\frac{1}{u^p}, \ u>0 \ \mbox{in} \ \Omega \subset \R^n.
\end{equation}

\begin{lem}
Assume that $p$ satisfies
\begin{equation}
n-2+ \frac{8}{p+1}> \frac{4}{p+1} (\frac{4}{p+1} + n-2)
\end{equation}
Let $u$ be a classical solution to (\ref{negative}) in $B_r (x) \subset B_R (x) \subset \Omega$. Then the following quantity

\begin{eqnarray*}
\tilde{ E}(r;x,u)&:=&r^{4\frac{p-1}{p+1}-n}
\int_{B_r (x)}\frac{1}{2}(\Delta u)^2-\frac{1}{p-1}u^{1-p}\\
&&-\frac{2}{p+1}\left(n-2+\frac{4}{p+1}\right)r^{-\frac{8}{p+1}+1-n}\int_{\partial
B_r (x)}u^2\\
&&-\frac{2}{p+1}\left(n-2+\frac{4}{p+1}\right)
\frac{d}{d r}\left(r^{-\frac{8}{p+1}+2-n}\int_{\partial
B_r (x)}u^2\right)\\
&&+\frac{r^3}{2}\frac{d}{d r}\left[ r^{-\frac{8}{p+1}+1-n}\int_{\partial
B_r (x)}\left(-\frac{4}{p+1}r^{-1}u+\frac{\partial
u}{\partial r}\right)^2\right]\\
&&+\frac{1}{2}\frac{d}{d r}\left[r^{-\frac{8}{p+1}+4-n}\int_{\partial
B_r (x)}\left(|\nabla u|^2-|\frac{\partial u}{\partial
r}|^2\right)\right]\\
&&+\frac{1}{2}r^{-\frac{8}{p+1}+3-n}\int_{\partial
B_r (x)}\left(|\nabla u|^2-|\frac{\partial u}{\partial
r}|^2\right)
\end{eqnarray*}
is increasing in $r$. Furthermore there exists $c_0>0$ such that
\begin{equation}
\frac{d}{dr}E(r;0,u)\geq c_0 r^{-n+2-\frac{8}{p+1}}\int_{\partial
B_r}\left(-\frac{4}{p+1}r^{-1}u+\frac{\partial u}{\partial
r}\right)^2.
\end{equation}

\end{lem}
In the rest of the paper, sometimes we use $E(r;x)$ or $E(r)$ if no confusion occurs.

\section{Homogeneous solutions}
\setcounter{equation}{0}

For the applications below, we give a non-existence result for
homogeneous stable solution of \eqref{equation}. (This corresponds to the tangent cone analysis of Fleming.) By the
Hardy-Rellich inequality, this result is sharp.
\begin{thm}\label{thm homogeneous solution}
Let $u\in W^{2,2}_{loc}(\mathbb{R}^n\setminus\{0\})$ be a
homogeneous, stable solution of \eqref{equation} in
$\mathbb{R}^n\setminus\{0\}$, for $p\in(\frac{n+4}{n-4},p_c(n))$.
Assume that $|u|^{p+1}\in L^1_{loc}(\mathbb{R}^n\setminus\{0\})$.
Then $u\equiv 0$.
\end{thm}
\begin{proof}
There exists a $w\in W^{2,2}(\mathbb{S}^{n-1})$ such that in polar
coordinates
\[u(r,\theta)=r^{-\frac{4}{p-1}}w(\theta).\]
Since $u\in W^{2,2}(B_2\setminus B_1)\cap L^{p+1}(B_2\setminus
B_1)$, $w\in W^{2,2}(\mathbb{S}^{n-1})\cap
L^{p+1}(\mathbb{S}^{n-1})$.

Direct calculations show that $w$ satisfies (in
$W^{2,2}(\mathbb{S}^{n-1})$ sense)
\begin{equation}\label{equation on sphere}
\Delta_\theta^2 w-J_1\Delta_\theta w+J_2 w=w^p,
\end{equation}
where
\[J_1=\left(\frac{4}{p-1}+2\right)\left(n-4-\frac{4}{p-1}\right)+\frac{4}{p-1}\left(n-2-\frac{4}{p-1}\right),\]
\[J_2=\frac{4}{p-1}\left(\frac{4}{p-1}+2\right)\left(n-4-\frac{4}{p-1}\right)\left(n-2-\frac{4}{p-1}\right).\]

Because $w\in W^{2,2}(\mathbb{S}^{n-1})$, we can test
\eqref{equation on sphere} with $w$, and we get
\begin{equation}\label{3.1}
\int_{\mathbb{S}^{n-1}}|\Delta_\theta w|^2+J_1|\nabla_\theta
w|^2+J_2w^2=\int_{\mathbb{S}^{n-1}}|w|^{p+1}.
\end{equation}

For any $\varepsilon>0$, choose an $\eta_\varepsilon\in
C_0^{\infty}((\frac{\varepsilon}{2},\frac{2}{\varepsilon}))$, such
that $\eta_\varepsilon\equiv 1$ in
$(\varepsilon,\frac{1}{\varepsilon})$, and
\[r|\eta_\varepsilon^{\prime}(r)|+r^2|\eta_\varepsilon^{\prime\prime}(r)|\leq 64~~\text{for all}~~r>0.\]
Because $w\in W^{2,2}(\mathbb{S}^{n-1})\cap
L^{p+1}(\mathbb{S}^{n-1})$,
$r^{-\frac{n-4}{2}}w(\theta)\eta_\varepsilon(r)$ can be approximated
by $C_0^\infty(B_{4/\varepsilon}\setminus B_{\varepsilon/4})$
functions in $W^{2,2}(B_{2/\varepsilon}\setminus
B_{\varepsilon/2})\cap L^{p+1}(B_{2/\varepsilon}\setminus
B_{\varepsilon/2})$. Hence in the stability condition for $u$ we are allowed to
choose a test function of the form
$r^{-\frac{n-4}{2}}w(\theta)\eta_\varepsilon(r)$. Note that
\begin{eqnarray*}
\Delta\left(r^{-\frac{n-4}{2}}w(\theta)\eta_\varepsilon(r)\right)&=&
-\frac{n(n-4)}{4}r^{-\frac{n}{2}}\eta_{\varepsilon}(r)w(\theta)
+r^{-\frac{n}{2}}\eta_{\varepsilon}(r)\Delta_\theta w(\theta)\\
&&-3r^{-\frac{n}{2}+1}\eta_{\varepsilon}^{\prime}(r)w(\theta)
+r^{-\frac{n}{2}+2}\eta_{\varepsilon}^{\prime\prime}(r)w(\theta).
\end{eqnarray*}
Substituting this into the stability condition for $u$, we get
\begin{eqnarray*}
&&p\left(\int_{\mathbb{S}^{n-1}}|w|^{p+1}d\theta\right)
\left(\int_0^{+\infty}r^{-1}\eta_\varepsilon(r)^2dr\right)\\
&\leq& \left(\int_{\mathbb{S}^{n-1}}\left(|\Delta_\theta
w|^2+\frac{n(n-4)}{2}|\nabla_\theta
w|^2+\frac{n^2(n-4)^2}{16}w^2\right)d\theta\right)
\left(\int_0^{+\infty}r^{-1}\eta_\varepsilon(r)^2dr\right)\\
&&+O[\left(\int_0^{+\infty}r\eta_{\varepsilon}^{\prime}(r)^2
+r^3\eta_{\varepsilon}^{\prime\prime}(r)^2+r^2|\eta_{\varepsilon}^{\prime}(r)|\eta_{\varepsilon}(r)
+r\eta_{\varepsilon}(r)|\eta_{\varepsilon}^{\prime\prime}(r)|dr\right)\\
&&\ \ \ \ \ \ \ \ \ \times
\left(\int_{\mathbb{S}^{n-1}}w(\theta)^2+|\nabla_\theta
w(\theta)|^2d\theta\right)].
\end{eqnarray*}
Note that
\[\int_0^{+\infty}r^{-1}\eta_\varepsilon(r)^2dr\geq|\log\varepsilon|,\]
\[\int_0^{+\infty}r\eta_{\varepsilon}^{\prime}(r)^2
+r^3\eta_{\varepsilon}^{\prime\prime}(r)^2+r^2|\eta_{\varepsilon}^{\prime}(r)|\eta_{\varepsilon}(r)
+r\eta_{\varepsilon}(r)|\eta_{\varepsilon}^{\prime\prime}(r)|dr\leq
C,\] for some constant $C$ independent of $\varepsilon$. By letting
$\varepsilon\to 0$, we obtain
\[p\int_{\mathbb{S}^{n-1}}|w|^{p+1}d\theta\leq
\int_{\mathbb{S}^{n-1}}|\Delta_\theta
w|^2+\frac{n(n-4)}{2}|\nabla_\theta w|^2+\frac{n^2(n-4)^2}{16}w^2.\]
Substituting \eqref{3.1} into this we get
\[\int_{\mathbb{S}^{n-1}}(p-1)|\Delta_\theta
w|^2+(pJ_1-\frac{n(n-4)}{2})|\nabla_\theta
w|^2+(pJ_2-\frac{n^2(n-4)^2}{16})w^2\leq 0.\] If
$\frac{n+4}{n-4}<p<p_c(n)$, then $p-1>0$, $pJ_1-\frac{n(n-4)}{2}>0$
and $pJ_2-\frac{n^2(n-4)^2}{16}>0$ (cf. p. 338 in \cite{G-G-S}), so
$w\equiv 0$ and then $u\equiv 0$.
\end{proof}
For applications in Section 6, we record the form of $E(R;0,u)$ for
a homogeneous solution $u$.
\begin{rmk}
\label{rmk E for homogeneous solutions} Suppose
$u(r,\theta)=r^{-\frac{4}{p-1}}w(\theta)$ is a homogeneous solution,
where $p>\frac{n+4}{n-4}$ and $w\in W^{2,2}(\mathbb{S}^{n-1})\cap
L^{p+1}(\mathbb{S}^{n-1})$. In this case, for any $r>0$,
\[\int_{B_r\setminus B_{r/2}}|\Delta u|^2+|u|^{p+1}\leq cr^{n-4\frac{p+1}{p-1}}.\]
Because $n-4\frac{p+1}{p-1}<0$, by choosing $r=2^{-i}R$ and summing
in $i$ from $1$ to $+\infty$, we see
\[\int_{B_R}|\Delta u|^2+|u|^{p+1}\leq cR^{n-4\frac{p+1}{p-1}},\]
which converges to $0$ as $R\to 0$. Hence for any $R>0$, $E(R;0,u)$
is well-defined and by the homogeneity, it equals $E(1;0,u)$. By
definition
\begin{eqnarray*}
E(1;0,u)&=&\int_{B_1}\frac{1}{2}(\Delta
u)^2-\frac{1}{p+1}|u|^{p+1}\\
&&+\frac{4}{p-1}\left(n-2-\frac{4}{p-1}\right)\int_{\partial
B_1}u^2+\int_{\partial B_1}|\nabla_\theta u|^2\\
&=&\left(\frac{1}{2}-\frac{1}{p+1}\right)\int_{B_1}|u|^{p+1}+\frac{1}{2}\int_{\partial
B_1}\left(\frac{\partial u}{\partial r}\Delta
u-u\frac{\partial\Delta u}{\partial
r}\right)\\
&&+\frac{4}{p-1}\left(n-2-\frac{4}{p-1}\right)\int_{\partial
B_1}u^2+\int_{\partial B_1}|\nabla_\theta u|^2.
\end{eqnarray*}
By noting that
\[\frac{\partial u}{\partial r}=-\frac{4}{p-1}r^{-1}u,\ \ \ \
\frac{\partial^2 u}{\partial
r^2}=\frac{4}{p-1}\left(\frac{4}{p-1}+1\right)r^{-2}u,\]
\[\frac{\partial \Delta u}{\partial r}=-\left(2+\frac{4}{p-1}\right)r^{-1}\Delta u,\ \ \ \
\Delta
u=\frac{4}{p-1}\left(\frac{4}{p-1}+2-n\right)r^{-2}u+r^{-2}\Delta_\theta
u,\] we get
\[E(1;0,u)
 =\left(\frac{1}{2}-\frac{1}{p+1}\right)\int_{B_1}|u|^{p+1}
=\frac{1}{n-4\frac{p+1}{p-1}}\left(\frac{1}{2}-\frac{1}{p+1}\right)\int_{\partial
B_1}|w|^{p+1}.\] Replacing $|u|^{p+1}$ by $(\Delta u)^2$, we also
have
\begin{eqnarray*}
E(1;0,u)&=&\left(\frac{1}{2}-\frac{1}{p+1}\right)\int_{B_1}(\Delta
u)^2+\frac{p-1}{p+1}\int_{\partial B_1}|\nabla_\theta u|^2\\
&&+\frac{4}{p+1}\left(n-2-\frac{4}{p-1}\right)\int_{\partial
B_1}u^2.
\end{eqnarray*}
\end{rmk}

\section{The blow down analysis}
\label{sect liouville stable} \setcounter{equation}{0}

In this section we use the blow-down analysis to prove the
Liouville theorem for stable solutions. Throughout this section $u$
always denotes a smooth stable solution of \eqref{equation} in
$\mathbb{R}^n$.
\begin{thm}\label{thm Liouville for stable}
Let $u$ be a smooth stable solution of \eqref{equation} on
$\mathbb{R}^n$. If $1<p<p_c(n)$, then $u\equiv 0$.
\end{thm}

The following lemma appears in \cite{Y-W} for positive solution. It remains valid for sign-changing solutions, see  also \cite{HHY}.

\begin{lem}\label{lem 4.2}
Let $u$ be a smooth stable solution of \eqref{equation} and let $v =
\Delta u$. Then for some $C$ we have
\begin{align}
\label{7a}
\begin{aligned}
\int_{\R^n} (v^2 + |u|^{p+1})  \eta^2 \leq C \int_{\R^n} u^2\left(
|\nabla(\Delta \eta)\cdot\nabla \eta| + (\Delta \eta)^2 +
|\Delta(|\nabla\eta|^2)| \right)dx
\\
+ C\int_{\R^n} | uv | |\nabla \eta|^2dx
\end{aligned}
\end{align}
 for all $\eta\in C_0^\infty(\R^n)$.
\end{lem}

\begin{proof}
For completeness we give the proof. We have the identity
\begin{align*}
\int_{\R^n} (\Delta^2 \xi) \xi \eta^2 dx &= \int_{\R^n}
(\Delta(\xi\eta))^2 +\int_{\R^n} (-4 (\nabla \xi\cdot\nabla\eta)^2 +
2\xi\Delta\xi|\nabla \eta|^2) dx
\\
&\qquad +\int_{\R^n} \xi^2 (2 \nabla (\Delta \eta) \cdot \nabla \eta
+ (\Delta\eta)^2)dx ,
\end{align*}
for $\xi\in C^4(\R^n)$ and $\eta\in C_0^\infty(\R^n)$, see for
example Lemma 2.3 in \cite{Y-W}.

Taking $\xi = u$ yields
\begin{align*}
\int_{\R^n} |u|^{p+1} \eta^2 dx &= \int_{\R^n} (\Delta(u \eta))^2
+\int_{\R^n} (-4 (\nabla u\cdot\nabla\eta)^2 + 2 u v |\nabla
\eta|^2) dx
\\
&\qquad +\int_{\R^n} u^2 (2 \nabla (\Delta \eta) \cdot \nabla \eta +
(\Delta\eta)^2)dx ,
\end{align*}
Using the stability inequality with $u \eta$ yields
$$
p \int_{\R^n} |u|^{p+1} \eta^2 dx \leq \int_{\R^n} (\Delta(u
\eta))^2.
$$
Therefore
\begin{align*}
\int_{\R^n} (|u|^{p+1}\eta^2 +(\Delta (u\eta) )^2 ) dx &\leq C
\int_{\R^n} (|\nabla u|^2 |\nabla\eta|^2  + |u v| |\nabla \eta|^2)
dx
\\
&\qquad + C\int_{\R^n} u^2 (| \nabla (\Delta \eta) \cdot \nabla
\eta| + (\Delta\eta)^2)dx .
\end{align*}
Using $\Delta (\eta u) =  v \eta + 2\nabla \eta\cdot\nabla u + \eta
\Delta u$ we obtain
\begin{align*}
\int_{\R^n} (|u|^{p+1} + v^2 )\eta^2 dx &\leq C \int_{\R^n} (|\nabla
u|^2 |\nabla\eta|^2  + |u v| |\nabla \eta|^2) dx
\\
&\qquad + C\int_{\R^n} u^2 (| \nabla (\Delta \eta) \cdot \nabla
\eta| + (\Delta\eta)^2)dx .
\end{align*}
But
\begin{align*}
2\int_{\R^n} |\nabla u|^2 |\nabla \eta|^2dx & = \int_{\R^n}
\Delta(u^2) |\nabla \eta|^2dx - 2 \int_{\R^n} u v |\nabla \eta|^2dx
\\
& = \int_{\R^n} u^2 \Delta( |\nabla \eta|^2)dx - 2 \int_{\R^n} u v
|\nabla \eta|^2dx ,
\end{align*}
and hence
\begin{align*}
\int_{\R^n} (|u|^{p+1} + v^2 )\eta^2 dx &\leq C\int_{\R^n} u^2 (|
\nabla (\Delta \eta) \cdot \nabla \eta| + (\Delta\eta)^2 +| \Delta(
|\nabla \eta|^2) |   )  dx
\\
&\qquad + C \int_{\R^n} |u v| |\nabla \eta|^2) dx .
\end{align*}
This proves \eqref{7a}
\end{proof}
\begin{coro}\label{coro 4.3}
There exists a constant $C$ such that
\begin{equation}\label{new estimate of p+1}
\int_{B_R(x)}v^2+|u|^{p+1}\leq CR^{-4}\int_{B_{2R}(x)\setminus
B_R(x)}u^2+CR^{-2}\int_{B_{2R}(x)\setminus B_R(x)}|uv|,
\end{equation}
and
\begin{equation}\label{estimate of p+1}
\int_{B_R(x)}v^2+|u|^{p+1}\leq CR^{n-4\frac{p+1}{p-1}}.
\end{equation}
for all $B_R(x)$.
\end{coro}
\begin{proof}
The first inequality is a direct consequence of \eqref{7a}, by
choosing a cut-off function $\eta\in C_0^\infty(B_{2R}(x))$, such
that $\eta\equiv 1$ in $B_R(x)$, and for  ${k\leq 3}$,
$|\nabla^k\eta|\leq\frac{1000}{R^{k}}$.

 Exactly the same argument as in
\cite{Y-W} or \cite{HHY} provides the second estimate.  For
completeness, we record the proof here. Replacing $\eta$ in
\eqref{7a} by $\eta^m$, where $m$ is a larger integer and $\eta$ is
a cut-off function as before. Then
\begin{eqnarray*}
\int|uv||\nabla\eta^m|^2&=&m^2\int_{B_{2R}(x)\setminus
B_R(x)}|uv|\eta^{2m-2}|\nabla\eta|^2\\
&\leq&\frac{1}{2C}\int v^2\eta^{2m}+C\int
u^2\eta^{2m-4}|\nabla\eta|^4.
\end{eqnarray*}
Substituting this into \eqref{7a}, we obtain
\begin{eqnarray*}
\int (v^2+|u|^{p+1})\eta^{2m}&\leq&CR^{-4}\int_{B_{2R}(x)}u^2\eta^{2m-4}\\
&\leq&CR^{-4}\left(\int_{B_{2R}(x)}|u|^{p+1}\eta^{(m-2)(p+1)}\right)^{\frac{2}{p+1}}R^{n(1-\frac{2}{p+1})}\\
&\leq&CR^{-4}\left(\int_{B_{2R}(x)}|u|^{p+1}\eta^{(m-2)(p+1)}\right)^{\frac{2}{p+1}}R^{n(1-\frac{2}{p+1})}.
\end{eqnarray*}
This gives \eqref{estimate of p+1}. Here we have used the fact
$\eta^{2m}\geq\eta^{(m-2)(p+1)}$ because $0\leq\eta\leq1$ and $m$ is
large.
\end{proof}
\begin{proof}[Proof of Theorem \ref{thm Liouville for stable} for $1<p\leq\frac{n+4}{n-4}$]
For $p<\frac{n+4}{n-4}$, we can let $R\to+\infty$ in \eqref{estimate
of p+1} to get $u\equiv 0$ directly. If $p=\frac{n+4}{n-4}$, this
gives
\[\int_{\R^n}v^2+|u|^{p+1}<+\infty.\]
So
\[\lim\limits_{R\to+\infty}\int_{B_{2R}(x)\setminus B_R(x)}v^2+|u|^{p+1}=0.\]
Then by \eqref{new estimate of p+1}, and noting that now
$n=4\frac{p+1}{p-1}$,
\begin{eqnarray*}
&&\int_{B_R(x)}v^2+|u|^{p+1}\leq CR^{-4}\int_{B_{2R}(x)\setminus
B_R(x)}u^2+C\int_{B_{2R}(x)\setminus B_R(x)}|v|^2\\
%&\leq&CR^{-4}\int_{B_{2R}(x)\setminus
%B_R(x)}u^2+C\int_{B_{2R}(x)\setminus B_R(x)}|v|^2\\
&\leq&CR^{-4}\left(\int_{B_{2R}(x)\setminus
B_R(x)}|u|^{p+1}\right)^{\frac{2}{p+1}}R^{n(1-\frac{2}{p+1})}+C\int_{B_{2R}(x)\setminus B_R(x)}|v|^2\\
&\leq&C\left(\int_{B_{2R}(x)\setminus
B_R(x)}|u|^{p+1}\right)^{\frac{2}{p+1}}+C\int_{B_{2R}(x)\setminus
B_R(x)}|v|^2.
\end{eqnarray*}
This goes to $0$ as $R\to+\infty$, and we still get $u\equiv 0$.
\end{proof}
Next we concentrate on the case $p>\frac{n+4}{n-4}$. We first use
\eqref{estimate of p+1} to show
\begin{lem}\label{lem upper bound on E}
$\lim\limits_{r\to+\infty}E(r;0,u)<+\infty.$
\end{lem}
\begin{proof}
Since $E(r)$ is non-decreasing in $r$, we have
\[E(r)\leq\frac{1}{r}\int_r^{2r}E(t)dt
\leq\frac{1}{r^2}\int_r^{2r}\int_t^{t+r}E(\lambda)d\lambda dt.\]
 By \eqref{estimate of p+1},
\[\frac{1}{r^2}\int_r^{2r}\int_t^{t+r}\left(\lambda^{4\frac{p+1}{p-1}-n}
\int_{B_\lambda}\frac{1}{2}(\Delta
u)^2-\frac{1}{p+1}|u|^{p+1}\right)d\lambda dt\leq C.\]

 Next
\begin{eqnarray*}
&&\frac{1}{r^2}\int_r^{2r}\int_t^{t+r}\left(\lambda^{\frac{8}{p-1}+1-n}\int_{\partial
B_\lambda}u^2\right)d\lambda
dt\\
&=&\frac{1}{r^2}\int_r^{2r}\int_{B_{t+r}\setminus
B_t}|x|^{\frac{8}{p-1}+1-n}u(x)^2dxdt\\
&\leq&\frac{1}{r^2}\int_r^{2r} \left(\int_{B_{3r\setminus
B_r}}|x|^{(\frac{8}{p-1}+1-n)\frac{p+1}{p-1}}\right)^{\frac{p-1}{p+1}}
\left(\int_{B_{3r}}|u(x)|^{p+1}\right)^{\frac{2}{p+1}}dt\\
&\leq&C.
\end{eqnarray*}
The same estimate holds for the term in $E(r)$ containing
\[\int_{\partial B_\lambda}\left(|\nabla
u|^2-|\frac{\partial u}{\partial r}|^2\right).\] For this we need to
note the following estimate
\[\int_{B_r}|\nabla u|^2\leq Cr^{2}\int_{B_{2r}}(\Delta u)^2+Cr^{-2+n\frac{p-1}{p+1}}
\left(\int_{B_{2r}}|u|^{p+1}\right)^{\frac{2}{p+1}}\leq
Cr^{n-\frac{8}{p-1}-2}.\]

Now consider
\begin{eqnarray*}
&&\frac{1}{r^2}\int_r^{2r}\int_t^{t+r}\frac{\lambda^3}{2}\frac{d}{d\lambda}\left[\lambda^{\frac{8}{p-1}+1-n}\int_{\partial
B_\lambda}\left(\frac{4}{p-1}\lambda^{-1}u+\frac{\partial
u}{\partial r}\right)^2\right]d\lambda
dt\\
&=&\frac{1}{2r^2}\int_r^{2r}\{(t+r)^{\frac{8}{p-1}+4-n}\int_{\partial
B_{t+r}}\left(\frac{4}{p-1}(t+r)^{-1}u+\frac{\partial u}{\partial
r}\right)^2\\
&&\ \ \ \ \ \ \ \ \ \ \ \ -t^{\frac{8}{p-1}+4-n}\int_{\partial
B_{t}}\left(\frac{4}{p-1}t^{-1}u+\frac{\partial u}{\partial
r}\right)^2\}dt\\
&&-\frac{3}{2r^2}\int_r^{2r}\int_t^{t+r}\lambda^{\frac{8}{p-1}+3-n}\int_{\partial
B_\lambda}\left(\frac{4}{p-1}\lambda^{-1}u+\frac{\partial
u}{\partial r}\right)^2d\lambda dt\\
&\leq&\frac{C}{r^2}\int_{B_{3r}\setminus
B_r}|x|^{\frac{8}{p-1}+4-n}\left(\frac{4}{p-1}|x|^{-1}u+\frac{\partial
u}{\partial r}\right)^2\\
&\leq&C.
\end{eqnarray*}
The remaining terms in $E(r)$ can be treated similarly.
\end{proof}

For any $\lambda>0$, define
\[u^\lambda(x):=\lambda^{\frac{4}{p-1}}u(\lambda x),\ \ \ \ \
v^\lambda(x):=\lambda^{\frac{4}{p-1}+2}v(\lambda x).\] $u^\lambda$
is also a smooth stable solution of \eqref{equation} on
$\mathbb{R}^n$.

By rescaling \eqref{estimate of p+1}, for all $\lambda>0$ and balls
$B_r(x)\subset\mathbb{R}^n$,
\[\int_{B_r(x)}(v^\lambda)^2+|u^\lambda|^{p+1}\leq Cr^{n-4\frac{p+1}{p-1}}.\]
In particular, $u^\lambda$ are uniformly bounded in
$L^{p+1}_{loc}(\mathbb{R}^n)$. By elliptic estimates, $u^\lambda$
are also uniformly bounded in $W^{2,2}_{loc}(\mathbb{R}^n)$. Hence,
up to a subsequence of $\lambda\to+\infty$, we can assume that
$u^\lambda\to u^\infty$ weakly in $W^{2,2}_{loc}(\mathbb{R}^n)\cap
L^{p+1}_{loc}(\mathbb{R}^n)$. By compactness embedding for Sobolev
functions, $u^\lambda\to u^\infty$ strongly in
$W^{1,2}_{loc}(\mathbb{R}^n)$. Then for any ball $B_R(0)$, by
interpolation between $L^q$ spaces and noting \eqref{estimate of
p+1}, for any $q\in[1,p+1)$, as $\lambda\to+\infty$,
\begin{equation}\label{convergence in q<p+1}
\|u^\lambda-u^\infty\|_{L^{q}(B_R(0))}\leq \|u^\lambda-
u^\infty\|_{L^1(B_R(0))}^t\|u^\lambda-
u^\infty\|_{L^{p+1}(B_R(0))}^{1-t}\to 0,
\end{equation}
where $t\in(0,1)$ satisfies $\frac{1}{q}=t+\frac{1-t}{p+1}$. That
is, $u^\lambda\to u^\infty$ in $L^{q}_{loc}(\mathbb{R}^n)$ for any
$q\in [1,p+1)$.

 For any function $\varphi\in
C_0^\infty(\mathbb{R}^n)$,
\[\int_{\mathbb{R}^n}\Delta u^\infty\Delta\varphi-(u^\infty)^p\varphi
=\lim\limits_{\lambda\to+\infty}\int_{\mathbb{R}^n}\Delta
u^\lambda\Delta\varphi-(u^\lambda)^p\varphi=0.\]
\[\int_{\mathbb{R}^n}(\Delta\varphi)^2-p(u^\infty)^{p-1}\varphi^2
=\lim\limits_{\lambda\to+\infty}\int_{\mathbb{R}^n}(\Delta\varphi)^2-p(u^\lambda)^{p-1}\varphi^2\geq0.\]
 Thus $u^\infty\in W^{2,2}_{loc}(\mathbb{R}^n)\cap L^{p+1}_{loc}(\mathbb{R}^n)$ is a stable
solution of \eqref{equation} in $\mathbb{R}^n$.

\begin{lem}
$u^\infty$ is homogeneous.
\end{lem}
\begin{proof}
For any $0<r<R<+\infty$, by the monotonicity of $E(r;0,u)$ and Lemma
\ref{lem upper bound on E},
\[
\lim\limits_{\lambda\to+\infty}E(\lambda R;0,u)-E(\lambda
r;0,u)=0.\] Then by the scaling invariance of $E$ and applying
Corollary \ref{coro characterization of homogeneous solutions}, we
see
\begin{eqnarray*}
0
&=&\lim\limits_{\lambda\to+\infty}E(R;0,u^\lambda)-E(r;0,u^\lambda)\\
&\geq&c(n,p)\lim\limits_{\lambda\to+\infty}\int_{B_R\setminus
B_r}\frac{\left(\frac{4}{p-1}|x|^{-1}u^\lambda(x) +\frac{\partial
u^\lambda}{\partial r}(x)\right)^2}{|x|^{n-2-\frac{8}{p-1}}}dx\\
&\geq&c(n,p)\int_{B_R\setminus
B_r}\frac{\left(\frac{4}{p-1}|x|^{-1}u^\infty(x) +\frac{\partial
u^\infty}{\partial r}(x)\right)^2}{|x|^{n-2-\frac{8}{p-1}}}dx.
\end{eqnarray*}
Note that in the last inequality we only used the weak convergence
of $u^\lambda$ to $u^\infty$ in $W^{1,2}_{loc}(\mathbb{R}^n)$. Now
\[\frac{4}{p-1}r^{-1}u^\infty
+\frac{\partial u^\infty}{\partial
r}=0,a.e.~~\text{in}~~\mathbb{R}^n.\] Integrating in $r$ shows that
\[u^\infty(x)=|x|^{-\frac{4}{p-1}}u^\infty(\frac{x}{|x|}).
\]
That is, $u^\infty$ is homogeneous.
\end{proof}
By Theorem \ref{thm homogeneous solution}, $u^\infty\equiv 0$. Since
this holds for the limit of any sequence $\lambda\to+\infty$, by
\eqref{convergence in q<p+1} we get
\[\lim\limits_{\lambda\to+\infty}u^\lambda=0\ \ \ \text{strongly
in}\ L^2(B_4(0)).\]
 Now we show
\begin{lem}\label{lem convergence of E}
$\lim\limits_{r\to+\infty}E(r;0,u)=0$.
\end{lem}
\begin{proof}
For all $\lambda\to+\infty$,
\[\lim\limits_{\lambda\to+\infty}\int_{B_4(0)}(u^\lambda)^2=0.\]
Because $v^\lambda$ are uniformly bounded in $L^2(B_4(0))$, by the
Cauchy inequality we also have
\[\lim\limits_{\lambda\to+\infty}\int_{B_4(0)}|u^\lambda v^\lambda|\leq\lim\limits_{\lambda\to+\infty}
\left(\int_{B_4(0)}(u^\lambda)^2\right)^{\frac{1}{2}}\left(\int_{B_4(0)}(v^\lambda)^2\right)^{\frac{1}{2}}=0.\]
 By \eqref{new estimate of p+1},
\begin{eqnarray}\label{convergence of energy}
\lim\limits_{\lambda\to+\infty}\int_{B_3(0)}(v^\lambda)^2+|u^\lambda|^{p+1}
&\leq&C\lim\limits_{\lambda\to+\infty}\left(\int_{B_4(0)}(u^\lambda)^2+\int_{B_4(0)}\big|u^\lambda
v^\lambda\big|\right)\nonumber\\&=&0.
\end{eqnarray}
By the interior $L^2$ estimate, we get
\[\lim\limits_{\lambda\to+\infty}\int_{B_2(0)}\sum_{k\leq 2}|\nabla^ku^\lambda|^2=0.\]
In particular, we can choose a sequence $\lambda_i\to+\infty$ such
that
\[\int_{B_2(0)}\sum_{k\leq 2}|\nabla^ku^{\lambda_i}|^2\leq 2^{-i}.\]
By this choice we have
\[\int_1^2\sum_{i=1}^{+\infty}\int_{\partial B_r}\sum_{k\leq
2}|\nabla^ku^{\lambda_i}|^2dr\leq\sum_{i=1}^{+\infty}\int_1^2\int_{\partial
B_r}\sum_{k\leq 2}|\nabla^ku^{\lambda_i}|^2dr\leq 1.\] That is, the
function
\[f(r):=\sum_{i=1}^{+\infty}\int_{\partial B_r}\sum_{k\leq
2}|\nabla^ku^{\lambda_i}|^2\in L^1((1,2)).\] There exists an
$r_0\in(1,2)$ such that $f(r_0)<+\infty$. From this we get
\[\lim\limits_{i\to+\infty}\|u^{\lambda_i}\|_{W^{2,2}(\partial B_{r_0})}=0.\]
Combining this with \eqref{convergence of energy} and the scaling
invariance of $E(r)$, we get
\[\lim\limits_{i\to+\infty}E(\lambda_ir_0;0,u)=\lim\limits_{i\to+\infty}E(r_0;0,u^{\lambda_i})=0.\]
Since $\lambda_ir_0\to+\infty$ and $E(r;0,u)$ is non-decreasing in
$r$, we get \[\lim\limits_{\lambda\to+\infty}E(r;0,u)=0.\qedhere\]
\end{proof}
By the smoothness of $u$, $\lim\limits_{r\to 0}E(r;0,u)=0$. Then
again by the monotonicity of $E(r;0,u)$ and the previous lemma, we
obtain
\[E(r;0,u)=0~~\text{for all}~~r>0.\]
Then again by Corollary \ref{coro characterization of homogeneous
solutions}, $u$ is homogeneous, and then $u\equiv 0$ by Theorem
\ref{thm homogeneous solution} (or by the smoothness of $u$). This
finishes the proof of Theorem \ref{thm Liouville for stable}.

\section{Finite Morse Index Solutions}
\setcounter{equation}{0}

In this section we prove Theorem~\ref{thmstable}.
%\begin{thm}
%%\label{thm Liouville for finite Morse}
%Let $u$ be a positive smooth solution of \eqref{equation} on
%$\mathbb{R}^n$, which is stable outside a compact set, if
%$p\in(\frac{n+4}{n-4},p_c(n))$, then $u\equiv 0$.
%\end{thm}
First, by the doubling lemma \cite{P-Q-S} and our Liouville theorem
for stable solutions Theorem \ref{thm Liouville for stable}, we have
\begin{lem}\label{finite Morse decay bound}
Let $u$ be a finite Morse index (positive or sign changing) solution of (\ref{equation}). There exists a constant $C_1$ and $R_0$ such that for all $x\in
B_{R_0}(0)^c$,
\[|u(x)|\leq C|x|^{-\frac{4}{p-1}}.\]
\end{lem}
\begin{proof}
Assume that $u$ is stable outside $B_{R_0}$. For $x\in B_{R_0}^c$,
let $M(x)=|u(x)|^{\frac{p-1}{4}}$ and $d(x)=|x|-R_0$, the distance
to $B_{R_0}$. Assume that there exists a sequence of $x_k\in
B_{R_0}^c$ such that
\begin{equation}\label{5.1}
M(x_k)d(x_k)\geq 2k.
\end{equation}
Since $u$ is bounded on any compact set of $\mathbb{R}^n$,
$d(x_k)\to+\infty$.

By the doubling lemma \cite{P-Q-S}, there exists another sequence
$y_k\in B_{R_0}^c$, such that
\begin{enumerate}
\item $M(y_k)d(y_k)\geq 2k$;
\item $M(y_k)\geq M(x_k)$;
\item $M(z)\leq 2M(y_k)$ for any $z\in B_{R_0}^c$ such that
$|z-y_k|\leq\frac{k}{M(y_k)}$.
\end{enumerate}
Now define
\[u_k(x)=M(y_k)^{-\frac{4}{p-1}}u(y_k+M(y_k)^{-1}x),~~\text{for}~~x\in B_k(0).\]
By definition, $|u_k(0)|=1$. By (3), $|u_k|\leq2^{\frac{p-1}{4}}$ in
$B_k(0)$. By (1), $B_{k/M(y_k)}(y_k)\cap B_{R_0}=\emptyset$, which
implies that $u$ is stable in $B_{k/M(y_k)}(y_k)$. Hence $u_k$ is
stable in $B_k(0)$.

By elliptic regularity, $u_k$ are uniformly bounded in
$C^5(B_k(0))$. Up to a subsequence, $u_k$ converges to $u_\infty$ in
$C^4_{loc}(\mathbb{R}^n)$. By the above conditions on $u_k$, we have
\begin{enumerate}
\item $|u_\infty(0)|=1$;
\item $|u_\infty|\leq2^{\frac{p-1}{4}}$ in $\mathbb{R}^n$;
\item $u_\infty$ is a smooth stable solution of \eqref{equation}
in $\mathbb{R}^n$.
\end{enumerate}
By the Liouville theorem for stable solutions, Theorem \ref{thm
Liouville for stable}, $u_\infty\equiv0$. This is a contradiction,
so \eqref{5.1} does not hold.
\end{proof}
\begin{coro}\label{corococo}
There exists a constant $C_1$ and $R_0$ such that for all $x\in
B_{3R_0}(0)^c$,
\begin{equation}\label{5.2}
\sum_{k\leq 3}|x|^{\frac{4}{p-1}+k}|\nabla^k u(x)|\leq C_3.
\end{equation}
\end{coro}
\begin{proof}
For any $x_0$ with $|x_0|>3R_0$, take
$\lambda=\frac{|x_0|}{2}$ and define
\[\bar{u}(x)=\lambda^{\frac{4}{p-1}}u(x_0+\lambda x).\]
By the previous lemma, $|\bar{u}|\leq C_1$ in $B_1(0)$. Standard
elliptic estimates give
\[\sum_{k\leq 3}|\nabla^k \bar{u}(0)|\leq C_3.\]
Rescaling back we get \eqref{5.2}.
\end{proof}

\begin{rmk}By the same proof of Lemma \ref{finite Morse decay bound} and Corollary \ref{corococo}, one easily obtains the second part of Theorem \ref{thm full regularity2}.
\end{rmk}

\subsection{The subcritical case $1<p<\frac{n+4}{n-4}$}

We use the following Pohozaev identity. For its proof, see \cite{P-S
1, P-S 2}.
\begin{lem}
\begin{align}\label{Pohozaev identity}
\int_{B_R}\frac{n-4}{2}(\Delta
u)^2-\frac{n}{p+1}|u|^{p+1}=\\\nonumber
\int_{\partial B_R}\frac{R}{2}(\Delta
u)^2+\frac{R}{p+1}|u|^{p+1}+R\frac{\partial u}{\partial
r}\frac{\partial \Delta u}{\partial r}-\Delta u\frac{\partial
(x\cdot\nabla u)}{\partial r}.
\end{align}
\end{lem}
By taking $R\to+\infty$ and using \eqref{5.2}, and noting that
$p<\frac{n+4}{n-4}$, we see that
\[\int_{\partial B_R}\frac{R}{2}(\Delta
u)^2+\frac{R}{p+1}|u|^{p+1}+R\frac{\partial u}{\partial
r}\frac{\partial \Delta u}{\partial r}-\Delta u\frac{\partial
(x\cdot\nabla u)}{\partial r}\to 0.\]
 By \eqref{5.2}, we also have
\[(\Delta u)^2+|u|^{p+1}\leq C(1+|x|)^{-4\frac{p+1}{p-1}}.\]
Since $p<\frac{n+4}{n-4}$, $4\frac{p+1}{p-1}>n$. Hence
\[
\int_{\mathbb{R}^n}(\Delta u)^2+|u|^{p+1}<+\infty.
\]
 Taking limit in \eqref{Pohozaev identity}, we get
\begin{equation}\label{5.3}
\int_{\R^n}\frac{n-4}{2}(\Delta u)^2-\frac{n}{p+1}|u|^{p+1}=0.
\end{equation}
Take an $\eta\in C_0^\infty(B_{2})$, $\eta\equiv 1$ in $B_1$ and
$\sum_{k\leq 2}|\nabla^k\eta|\leq 1000$, and denote
$\eta_R(x)=\eta(x/R)$. By testing the equation \eqref{equation} with
$u(x)\eta_R^2$, we get
\[\int_{\R^n}(\Delta u)^2\eta_R^2-|u|^{p+1}\eta_R^2=-\int_{\R^n}2\nabla u\nabla\eta_R^2+u\Delta\eta_R^2.\]
By the same reasoning as above, we get
\[\int_{\R^n}(\Delta u)^2-|u|^{p+1}=0.\]
Substituting \eqref{5.3} into this, we get
\[\left(\frac{n-4}{2}-\frac{n}{p+1}\right)\int_{\R^n}|u|^{p+1}=0.\]
Since $\frac{n-4}{2}-\frac{n}{p+1}<0$, $u\equiv 0$.

\subsection{The critical case}

Since $u$ is stable outside $B_{R_0}$, Lemma \ref{lem 4.2} still
holds if the support of $\eta$ is outside $B_{R_0}$. Take
$\varphi\in C_0^\infty(B_{2R}\setminus B_{2R_0})$, such that
$\varphi\equiv 1$ in $B_{R}\setminus B_{3R_0}$ and $\sum_{k\leq
3}|x|^k|\nabla^k\varphi|\leq 100$. Then by choosing
$\eta=\varphi^m$, where $m$ is large, in \eqref{7a}, and by the same
reasoning to derive \eqref{estimate of p+1}, we get
\[\int_{B_R\setminus B_{3R_0}}(\Delta u)^2+|u|^{p+1}\leq C.\]
Letting $R\to+\infty$, we get
\[\int_{\R^n}(\Delta u)^2+|u|^{p+1}<+\infty.\]
This then implies that
\[\lim\limits_{R\to+\infty}\int_{B_{2R}\setminus B_R}R^{-1}|\nabla u|+R^{-2}|u|=0.\]
 Then we can proceed as in the subcritical case to
prove that
\[\int_{\R^n}(\Delta u)^2-|u|^{p+1}=0.\]

\subsection{The supercritical case}

Now we consider the case $p>\frac{n+4}{n-4}$.
\begin{lem}\label{lem 5.3}
There exists a constant $C_2$, such that for all $r>3R_0$,
$E(r;0,u)\leq C_2$.
\end{lem}
\begin{proof}
Expanding those boundary integrals in $E(r;0,u)$ into a full
formulation involving the differentials of $u$ up to third order,
and substituting \eqref{5.2} into this formulation, we get
\begin{eqnarray*}
E(r;0,u)&\leq&Cr^{4\frac{p+1}{p-1}-n}\left( \int_{B_r}(\Delta
u)^2+|u|^{p+1}\right)+C r^{\frac{8}{p-1}+1-n}\int_{\partial B_r}u^2\\
&&+Cr^{\frac{8}{p-1}+2-n}\int_{\partial B_r}|u||\nabla u|
+Cr^{\frac{8}{p-1}+3-n}\int_{\partial
B_r}|\nabla u|^2 \\
&&+Cr^{\frac{8}{p-1}+4-n}\int_{\partial B_r}|\nabla u||\nabla^2
u|\\
&\leq&C.
\end{eqnarray*}
This constant only depends on the constant in \eqref{5.2}.
\end{proof}
By Corollary \ref{coro characterization of homogeneous solutions},
we get
\begin{coro}\label{coro 5.3}
\[\int_{B_{3R_0}^c}\frac{\left(\frac{4}{p-1}|x|^{-1}u(x)
+\frac{\partial u}{\partial
r}(x)\right)^2}{|x|^{n-2-\frac{8}{p-1}}}dx<+\infty.\]
\end{coro}
As in the proof for stable solutions, define the blowing down
sequence
\[u^\lambda(x)=\lambda^{\frac{4}{p-1}}u(\lambda x).\]
By Lemma \ref{finite Morse decay bound}, $u^\lambda$ are uniformly
bounded in $C^5(B_r(0)\setminus B_{1/r}(0))$ for any fixed $r>1$.
$u^\lambda$ is stable outside $B_{R_0/\lambda}(0)$. There exists a
function $u^\infty\in C^4(\mathbb{R}^n\setminus\{0\})$, such that up
to a subsequence of $\lambda\to+\infty$, $u^\lambda$ converges to
$u^\infty$ in $C^4_{loc}(\mathbb{R}^n\setminus\{0\})$. $u^\infty$ is
a stable solution of \eqref{equation} in
$\mathbb{R}^n\setminus\{0\}$.

For any $r>1$, by Corollary \ref{coro 5.3},
\begin{eqnarray*}
&&\int_{B_r\setminus
B_{1/r}}\frac{\left(\frac{4}{p-1}|x|^{-1}u^\infty(x) +\frac{\partial
u^\infty}{\partial r}(x)\right)^2}{|x|^{n-2-\frac{8}{p-1}}}dx\\&=&
\lim\limits_{\lambda\to+\infty}\int_{B_r\setminus
B_{1/r}}\frac{\left(\frac{4}{p-1}|x|^{-1}u^\lambda(x)
+\frac{\partial u^\lambda}{\partial
r}(x)\right)^2}{|x|^{n-2-\frac{8}{p-1}}}dx\\
&=&\lim\limits_{\lambda\to+\infty}\int_{B_{\lambda r}\setminus
B_{\lambda/r}}\frac{\left(\frac{4}{p-1}|x|^{-1}u(x) +\frac{\partial
u}{\partial r}(x)\right)^2}{|x|^{n-2-\frac{8}{p-1}}}dx\\
&=&0.
\end{eqnarray*}
Hence $u^\infty$ is homogeneous, and by Theorem \ref{thm homogeneous
solution}, $u^\infty\equiv0$. This holds for every limit of
$u^\lambda$ as $\lambda\to+\infty$, thus we have
\[\lim\limits_{x\to\infty}|x|^{\frac{4}{p-1}}|u(x)|=0.\]
Then as in the proof of Corollary \ref{corococo}, we get
\begin{lem}
\[\lim\limits_{x\to\infty}\sum_{k\leq 4}|x|^{\frac{4}{p-1}+k}|\nabla^ku(x)|=0.\]
\end{lem}

For any $\varepsilon>0$, take an $R$ such that for $|x|>R$,
\[\sum_{k\leq 4}|x|^{\frac{4}{p-1}+k}|\nabla^ku(x)|\leq \varepsilon.\]
Then for $r\gg R$,
\begin{eqnarray*}
E(r;0,u)&\leq&Cr^{4\frac{p+1}{p-1}-n}\int_{B_R(0)}\left[(\Delta
u)^2+|u|^{p+1}\right]+C\varepsilon
r^{4\frac{p+1}{p-1}-n}\int_{B_r(0)\setminus B_{R}(0)}|x|^{-4\frac{p+1}{p-1}}\\
&&+C\varepsilon
r^{4\frac{p+1}{p-1}+1-n}\int_{\partial B_r(0)}|x|^{-4\frac{p+1}{p-1}}\\
&\leq&C(R)(r^{4\frac{p+1}{p-1}-n}+\varepsilon).
\end{eqnarray*}
Since $4\frac{p+1}{p-1}-n<0$ and $\varepsilon$ can be arbitrarily
small, we get $\lim\limits_{r\to+\infty}E(r;0,u)=0$. Because
$\lim\limits_{r\to 0}E(r;0,u)=0$ (by the smoothness of $u$), the
same argument for stable solutions implies that $u\equiv 0$.
\begin{rmk}
The monotonicity formula approach here is in some sense equivalent
to the Pohohazev identity method (see for example \cite{Y-W}). The
convergence of $u^\lambda$ can also be seen by writing the equation
in exponential polar coordinates.
\end{rmk}

\section{Partial regularity in high dimensions}
\setcounter{equation}{0}

Here we study the partial regularity
for the extremal solution to the problem \eqref{boundary value problem}, and prove Theorems \ref{thm full regularity} and \ref{thm partial regularity}.
%\begin{equation}\label{boundary value problem}
%\left\{ \begin{aligned}
% &\Delta^2 u=\lambda(1+u)^p,~~\text{in}~~\Omega,\\
% &u>0,~~\text{in}~~\Omega,\\
% &u=\Delta u=0,~~\text{on}~~\partial\Omega,
%                          \end{aligned} \right.
%\end{equation}
%Here $\Omega$ is a smooth bounded domain in $\mathbb{R}^n$. It is
%well known that there exists a $\lambda^\ast>0$ depending only on
%$p$ and $\Omega$, such that
%\begin{enumerate}
%\item For $\lambda\in(0,\lambda^\ast)$, $\eqref{boundary value
%problem}$ has a minimal and classical solution $u_\lambda$, which is
%stable in $\Omega$;
%\item for $\lambda\in(0,\lambda^\ast)$, $u_\lambda$ is increasing in $\lambda$;
%\item $u_{\lambda^\ast}=\lim\limits_{\lambda\to\lambda^\ast}u_\lambda$ is a unique weak solution of \eqref{boundary value problem}
%for $\lambda^\ast$;
%\item No weak solution exists for $\lambda>\lambda^\ast$.
%\end{enumerate}
Recall that we defined $n_p$ to be the smallest dimension such that
Theorem \ref{thm homogeneous solution} does not hold. This is also
the smallest dimension such that the Liouville theorem for stable
solutions, Theorem \ref{thm Liouville for stable}, and the
classification result for stable homogeneous solutions, Theorem
\ref{thm homogeneous solution}, do not hold.
%\begin{thm}
%%\label{thm full regularity}
%If $n<n_p$, the extremal solution $u_{\lambda^\ast}$ is smooth.
%\end{thm}
\begin{proof}[Proof of Theorem \ref{thm full regularity}.]
For $0<\lambda<\lambda^\ast$ let $u_\lambda>0$ be the minimal solution of \eqref{boundary value problem}.
We claim that
\begin{equation}\label{6.1}
\sup_{\lambda\in(0,\lambda^\ast)}\|u_\lambda\|_{L^\infty(\Omega)}<+\infty.
\end{equation}
Then by elliptic estimates, as $\lambda\to\lambda^\ast$, $u_\lambda$
are uniformly bounded in $C^5(\overline{\Omega})$. Because
$u_\lambda$ converges to $u_{\lambda^\ast}$ pointwisely in $\Omega$,
$u_{\lambda^\ast}\in C^4(\overline{\Omega})$, and then we get
$u_{\lambda^\ast}\in C^\infty(\overline{\Omega})$ by bootstrapping
elliptic estimates.

To prove \eqref{6.1}, we use the classical blow up method of Gidas
and Spruck. Let $x_\lambda$ attain
$\max_{\overline{\Omega}}u_\lambda$, and assume that
\[L_\lambda=u_\lambda(x_\lambda)+1\to+\infty.\]
By the maximum principle, $x_\lambda\in\Omega$ is an interior point
and
\begin{equation}\label{6.3}
-\Delta u_\lambda>0~~~\text{in}~~~\Omega.
\end{equation}

Define
\[\bar{u}_\lambda=L_\lambda^{-1}\left(
u_{\lambda}(x_\lambda+L_\lambda^{-\frac{p-1}{4}}x)+1\right)~~\text{in}~~\Omega_\lambda,\]
where $\Omega_\lambda=L_\lambda^{-\frac{p-1}{4}}(\Omega-x_\lambda)$.
$\bar{u}_\lambda$ is a smooth stable solution of \eqref{equation} in
$\Omega_\lambda$, satisfying
\begin{equation}\label{6.2}
\bar{u}_\lambda(0)=\max_{\overline{\Omega_\lambda}}\bar{u}_\lambda=1,
\end{equation}
and the boundary condition
\[\bar{u}_\lambda=L_\lambda^{-1},\ \ \ \ \Delta \bar{u}_\lambda=0~~~\text{on}~~~\partial\Omega_\lambda.\]
From this, with the help of standard elliptic estimates, we see for
any $R>0$, $\bar{u}_\lambda$ are uniformly bounded in
$C^5(\Omega_\lambda\cap B_R(0))$. By rescaling \eqref{6.3},
\begin{equation}\label{6.4}
-\Delta \bar{u}_\lambda>0~~~\text{in}~~~\Omega_\lambda.
\end{equation}

Since $\Omega$ is a smooth domain, as $\lambda\to\lambda^\ast$,
$\Omega_\lambda$ either converges to $\mathbb{R}^n$ or to a half space
$H$. In the former case, $\bar{u}_\lambda$ converges (up
to a subsequence) to a limit $\bar{u}$ in $C^4_{loc}(\mathbb{R}^n)$.
Here $\bar{u}$ is a positive, stable, $C^4$ solution of
\eqref{equation} in $\mathbb{R}^n$. Then by Theorem \ref{thm
Liouville for stable}, $\bar{u}\equiv 0$. However, by passing to the limit
in \eqref{6.2}, we obtain
\[\bar{u}_\lambda(0)=1.\]
This is a contradiction.

If $\Omega_\lambda$ converges to a half space $H=\{ x_1>-h\}$ for
some $h>0$, $\bar{u}_\lambda$ converges (up to a subsequence) to a
limit $\bar{u}$ in $C^4_{loc}(\overline{H})$. Here $\bar{u}$ is a
positive, stable, $C^4$ solution of \eqref{equation} in $H$, with
the boundary conditions
\[\bar{u}=\Delta \bar{u}=0~~~\text{on}~~~\partial H.\]
By taking limits in \eqref{6.2} and \eqref{6.4}, we obtain
\begin{equation*}
\left\{ \begin{aligned}
 &-\Delta \bar{u}=\bar{v}>0,~~\text{in}~~H,\\
 &-\Delta \bar{v}=\bar{u}^p>0,~~\text{in}~~H,\\
 &\bar{u}(0)=\max_{\overline{H}}\bar{u}=1.
                          \end{aligned} \right.
\end{equation*}
By elliptic estimates, the last condition implies that $\bar{v}$ is
bounded in $H$. Then by [Theorem 2, \cite{D}] or [Theorem 10, \cite{Si}], $\frac{\partial \bar{u}}{\partial x_1} >0, \frac{\partial  \bar{v}}{\partial x_1} >0$.  Then the function $ w(y)= \lim_{x_1 \to +\infty}  \bar{u} (x_1, y)$ exists for all $y \in \R^{n-1}$ and satisfies $ \Delta^2 w= w^p$ in $\R^{n-1}$. By the arguments in Section 3 of \cite{Y-W} this function $w$ must be stable in $\R^{n-1}$ and non trivial. By Theorem \ref{thmstable}, $p \geq p_c (n-1)\geq p_c(n)$. This is impossible.

We conclude that $\bar{u}\equiv 0$, which is
a contradiction. This finishes the proof of \eqref{6.1}.
\end{proof}

%For $n\geq n_p$, we do not expect the extremal solution
%$u_{\lambda^\ast}$ to be regular as in the previous theorem. To
%study its regularity, we introduce
%\begin{defi}
%%\label{defintion}
%The singular set $S$ of a solution $u$ contain those points where in
%any neighborhood of this point $u$ is not bounded. Its complement is
%the regular set of $u$.
%\end{defi}
%By definition, the regular set is an open set. Now we state the
%interior partial regularity for $u_{\lambda^\ast}$.
%\begin{thm}\label{thm partial regularity}
%Let $u_{\lambda^\ast}$ be the extremal solution to \eqref{boundary
%value problem}, then the Hausdorff dimension of $S$ is no more than
%$n-n_p$. Moreover, when $n=n_p$, $S$ is a discrete set.
%\end{thm}

\bigskip

The remaining part is devoted to the proof of Theorem \ref{thm partial regularity}.
First we need the following lemma.
\begin{lem}
There exists a constant $C$, such that, for any ball
$B_{2r}(x)\subset\Omega$,
\begin{equation}\label{estimate of 2p by p+1}
r^{\frac{8p}{p-1}-n}\int_{B_{r}(x)}(u_{\lambda^\ast}+1)^{2p}\leq
Cr^{4\frac{p+1}{p-1}-n}\int_{B_{2r}(x)}(u_{\lambda^\ast}+1)^{p+1}+(\Delta
u_{\lambda^\ast})^2.
\end{equation}
\end{lem}
\begin{proof}
Denote $w_\lambda=u_\lambda+1$. By the maximum principle and Lemma
3.2 in \cite{CEG}, for any $\lambda\in(0,\lambda^\ast)$,
\[\Delta
w_\lambda\leq-\sqrt{\frac{2\lambda}{p+1}}w_\lambda^{\frac{p+1}{2}}<0~~\text{in}~~\Omega.\]
 Since $w_\lambda$ is smooth in $\Omega$, we can follow the proof in \cite{Y-W} to get Eq. (2.15) in \cite{Y-W}.
That is, for any $\eta\in C_0^\infty(\Omega)$,
\begin{eqnarray}\label{6.5}
\int_{\Omega}w_\lambda^{2p}\eta^2&\leq &\int_{\Omega}-\Delta
w_\lambda w_\lambda^p\left(|\nabla\eta|^2
+|\Delta\eta^2|\right)\\\nonumber &&+C\int_\Omega(\Delta
w_\lambda)^2\left[|\nabla\Delta\eta\nabla\eta|+|\Delta|\nabla\eta|^2|+|\Delta\eta|^2.\right]
\end{eqnarray}
Take $\varphi\in C_0^\infty(B_{2r}(x))$ such that $0\leq\varphi\leq
1$, $\varphi\equiv 1$ in $B_{r}(x)$ and
\[\sum_{k\leq 4}r^k|\nabla^k\varphi|\leq 1000.\]
Substituting $\eta=\varphi^m$ into \eqref{6.5} with $m$ larger, and
then using H\"{o}lder inequality (exactly as in the derivation of
Eq. (2.16) of \cite{Y-W}), we get \eqref{estimate of 2p by p+1} for
$u_\lambda$.

This implies that $u_\lambda$ are uniformly bounded in
$L^{2p}_{loc}(\Omega)$. By the interior $L^2$ estimate, $u_\lambda$
are also uniformly bounded in $W^{4,2}_{loc}(\Omega)$. By the same
proof of \eqref{convergence in q<p+1}, as $\lambda\to\lambda^\ast$,
$u_\lambda\to u_{\lambda^\ast}$ in $W^{3,2}_{loc}(\Omega)\cap
L^{p+1}_{loc}(\Omega)$. Then
\begin{eqnarray*}
r^{\frac{8p}{p-1}-n}\int_{B_{r}(x)}(u_{\lambda^\ast}+1)^{2p}&\leq&\lim\limits_{\lambda\to\lambda^\ast}
r^{\frac{8p}{p-1}-n}\int_{B_{r}(x)}(u_{\lambda}+1)^{2p}\\
&\leq&C\lim\limits_{\lambda\to\lambda^\ast}r^{4\frac{p+1}{p-1}-n}\int_{B_{2r}(x)}(u_{\lambda}+1)^{p+1}
+(\Delta u_{\lambda})^2\\
&\leq& Cr^{4\frac{p+1}{p-1}-n}\int_{B_{2r}(x)}
(u_{\lambda^\ast}+1)^{p+1}+(\Delta u_{\lambda^\ast})^2.
\end{eqnarray*}
Here we have used Fatou's lemma to deduce the first inequality.
\end{proof}
Below we denote $u=u_{\lambda^\ast}+1$. Inequality
\eqref{estimate of 2p by p+1} implies that
\begin{equation}\label{estimate of 2p}
\int_{B_r(x)}u^{2p}\leq Cr^{n-\frac{8p}{p-1}}.
\end{equation}
 for any
ball $B_r(x)\subset\Omega$, with the constant $C$ depending only on
$p$ and $\Omega$. See for example the derivation of Eq.(2.16) in
\cite{Y-W}. Similarly, $u$ also satisfies \eqref{estimate of p+1}
for any ball $B_R(x)\subset\Omega$. Estimate \eqref{estimate of 2p
by p+1} will play a crucial role in our proof of the
$\varepsilon$-regularity lemma. Note that both \eqref{estimate of 2p
by p+1} and \eqref{estimate of 2p} are invariant under the scaling
for \eqref{equation}. Theses two are also preserved under various
limits (The precise notion of limit will be given below).

To prove the partial regularity of $u$, first we need the following
improvement of decay estimate.
\begin{lem}
\label{lem improvement of decay} There exist two universal constants
$\varepsilon_0>0$ and $\theta\in(0,1)$, such that if $u$ is a positive stable
solution of \eqref{equation} satisfying the estimate \eqref{estimate
of 2p by p+1}, and
\[(2R)^{4\frac{p+1}{p-1}-n}\int_{B_{2R}} u^{p+1}+(\Delta u)^2=\varepsilon\leq\varepsilon_0.\]
Then
\[(\theta R)^{4\frac{p+1}{p-1}-n}\int_{B_{\theta R}} u^{p+1}+(\Delta u)^2\leq\frac{\varepsilon}{2}.\]
\end{lem}
\begin{proof}
By rescaling, we can assume $R=1$. By \eqref{estimate of 2p by
p+1}, we have
\begin{equation}
\int_{B_{3/2}} u^{2p}\leq C\int_{B_{2}} u^{p+1}+(\Delta u)^2\leq
C\varepsilon.
\end{equation}
By $L^2$ estimates applied to $u$,
\[\|u\|_{W^{4,2}(B_{4/3})}\leq C\left(\|u^p\|_{L^2(B_{3/2})}+\|u\|_{L^2(B_{3/2})}\right)
\leq C\varepsilon^{\frac{1}{p+1}}.\] We can choose an
$r_0\in(1,4/3)$ so that
\begin{equation}
\label{bound on r_0}
\|u\|_{W^{2,2}(\partial B_{r_0})}\leq C\varepsilon^{\frac{1}{p+1}}.
\end{equation}
Now take the decomposition $u=u_1+u_2$, where
\begin{equation*}
\left\{ \begin{aligned}
 &\Delta^2 u_1=u^p,  &~~in  &~~B_{r_0},\\
 & u_1=\Delta u_1=0,  ~~ ~~~~  &~~on  & ~~\partial B_{r_0}(0),
                          \end{aligned} \right.
\end{equation*}
and
\begin{equation*}
 \left\{ \begin{aligned}
 &\Delta^2 u_2=0,  &~~in  &~~B_{r_0},\\
 & u_2=u, \Delta u_2=\Delta u,  ~~ ~~~~  &~~on  & ~~\partial
 B_{r_0}(0).
                          \end{aligned} \right.
\end{equation*}
By this decomposition,
\[\int_{B_{r_0}}\Delta u_1\Delta u_2=0.\]
Hence
\[\int_{B_{r_0}}(\Delta u)^2=\int_{B_{r_0}}(\Delta
u_1)^2+\int_{B_{r_0}}(\Delta u_2)^2.\]
 In particular,
\begin{equation}\label{orthogonal}
\int_{B_{r_0}}(\Delta u_2)^2\leq C\varepsilon.
\end{equation}

 By elliptic estimates
for biharmonic functions and \eqref{bound on r_0}, we have
\begin{equation*}
%\label{bound on u_2}
\sup_{B_{1/2}}|u_2|\leq C\left(\int_{\partial B_{r_0}}u^2+(\Delta
u)^2\right)^{1/2}\leq C\varepsilon^{\frac{1}{p+1}}.
\end{equation*}
Since $\Delta u_2$ is harmonic, $(\Delta u_2)^2$ is subharmonic in
$B_{r_0}$. By the mean value inequality for subharmonic functions
and \eqref{orthogonal}, for any $r\in(0,r_0)$,
\begin{equation*}
%\label{energy bound on u_2}
r^{4\frac{p+1}{p-1}-n}\int_{B_r}(\Delta u_2)^2\leq
r^{4\frac{p+1}{p-1}}r_0^{-n}\int_{B_{r_0}}(\Delta u_2)^2\leq
Cr^{4\frac{p+1}{p-1}}\varepsilon.
\end{equation*}

For $u_1$, first by the Green function representation (cf. Section
4.2 in \cite{G-G-S}), we have
\begin{equation}
\|u_1\|_{L^1(B_{r_0})}\leq C\|u^p\|_{L^1(B_{r_0})}\leq
C\left(\int_{B_2}u^{p+1}\right)^{\frac{p}{p+1}}\leq
C\varepsilon^{\frac{p}{p+1}}.
\end{equation}
Then by $L^2$ estimates using \eqref{estimate of 2p}, we have
\[\|u_1\|_{W^{4,2}(B_{r_0})}\leq C\left(\|u^p\|_{L^2(B_{r_0})}+\|u_1\|_{L^1(B_{r_0})}\right)
\leq C\varepsilon^{\frac{1}{2}}.\] By Sobolev embedding theorem, we
have
\[\|u_1\|_{L^{\frac{2n}{n-8}}(B_{r_0})} \leq
C\varepsilon^{\frac{1}{2}}.\] Then an interpolation between $L^1$
and $L^{\frac{2n}{n-8}}$ gives
\[\|u_1\|_{L^2(B_{r_0})} \leq
C\varepsilon^{\frac{1}{2}+2\delta},\] where $\delta>0$ is a constant
depending only the dimension $n$.

Next, by interpolation between Sobolev spaces, we get
\[\|\Delta u_1\|_{L^2(B_{r_0})}\leq\varepsilon^{-\delta}\|u_1\|_{L^2(B_{r_0})}
+C\varepsilon^{\delta}\|\Delta^2 u_1\|_{L^2(B_{r_0})}\leq
C\varepsilon^{\frac{1}{2}+\delta}.\] Multiplying the equation of
$u_1$ by $u_1$ and integrate by parts, we get
\[\int_{B_{r_0}}u^pu_1=\int_{B_{r_0}}(\Delta u_1)^2\leq C\varepsilon^{1+2\delta}.\]

By convexity, there exists a constant depending only on $p$ such
that
\[u^{p+1}\leq C\left(u_1^{p+1}+u_2^{p+1}\right).\]
 For $r\in(0,1/2)$, which will be determined below,
\begin{eqnarray*}
r^{4\frac{p+1}{p-1}-n}\int_{B_r}u^{p+1}&\leq&Cr^{4\frac{p+1}{p-1}-n}\int_{B_r}u_1^{p+1}
+Cr^{4\frac{p+1}{p-1}-n}\int_{B_r}u_2^{p+1}\\
&\leq&Cr^{4\frac{p+1}{p-1}-n}\int_{B_r}(u+u_2)^{p}u_1+Cr^{4\frac{p+1}{p-1}}\sup_{B_r}|u_2|^{p+1}\\
&\leq&Cr^{4\frac{p+1}{p-1}-n}\int_{B_r}u^pu_1
+Cr^{4\frac{p+1}{p-1}-n}\int_{B_r}\varepsilon^{\frac{p}{p+1}}u_1+Cr^{4\frac{p+1}{p-1}}\varepsilon\\
&\leq&Cr^{4\frac{p+1}{p-1}-n}\int_{B_{r_0}}u^pu_1
+Cr^{4\frac{p+1}{p-1}-n}\int_{B_{r_0}}\varepsilon^{\frac{p}{p+1}}u_1+Cr^{4\frac{p+1}{p-1}}\varepsilon\\
&\leq&Cr^{4\frac{p+1}{p-1}-n}\varepsilon^{1+2\delta}+Cr^{4\frac{p+1}{p-1}-n}\varepsilon^{\frac{2p}{p+1}}
+Cr^{4\frac{p+1}{p-1}}\varepsilon.
\end{eqnarray*}
For $(\Delta u)^2$, we have
\begin{eqnarray*}
r^{4\frac{p+1}{p-1}-n}\int_{B_r}(\Delta
u)^2&\leq&Cr^{4\frac{p+1}{p-1}-n}\int_{B_r}(\Delta u_1)^2
+Cr^{4\frac{p+1}{p-1}-n}\int_{B_r}(\Delta u_2)^2\\
&\leq&Cr^{4\frac{p+1}{p-1}-n}\int_{B_{r_0}}(\Delta u_1)^2
+Cr^{4\frac{p+1}{p-1}}r_0^{-n}\int_{B_{r_0}}(\Delta u_2)^2\\
&\leq&Cr^{4\frac{p+1}{p-1}-n}\varepsilon^{1+2\delta}
+Cr^{4\frac{p+1}{p-1}}\varepsilon.
\end{eqnarray*}
Putting these two together, we get
\[
r^{4\frac{p+1}{p-1}-n}\int_{B_r}(\Delta u)^2+u^{p+1} \leq
Cr^{4\frac{p+1}{p-1}-n}\varepsilon^{1+2\delta}+Cr^{4\frac{p+1}{p-1}-n}\varepsilon^{\frac{2p}{p+1}}
+Cr^{4\frac{p+1}{p-1}}\varepsilon.\]

We first choose $r=\theta\in(0,1/2)$ so that
\[C\theta^{4\frac{p+1}{p-1}}\leq\frac{1}{4}.\]
Then choose an $\varepsilon_0$ so that for every
$\varepsilon\in(0,\varepsilon_0)$,
\[C\theta^{4\frac{p+1}{p-1}-n}\varepsilon^{1+2\delta}
+C\theta^{4\frac{p+1}{p-1}-n}\varepsilon^{\frac{2p}{p+1}}\leq\frac{1}{4}\varepsilon.\]
By this choice we finish the proof.
\end{proof}
\begin{rmk}
Lemma
\ref{lem improvement of decay} also holds for a sign-changing solution $u$ of \eqref{equation} if it satisfies
\begin{equation}
\label{estimate of 2p by p+1 sign ch}
r^{\frac{8p}{p-1}-n}\int_{B_{r}(x)} |u|^{2p}\leq
Cr^{4\frac{p+1}{p-1}-n}\int_{B_{2r}(x)} |u|^{p+1}+(\Delta
u)^2 ,
\end{equation}
for any ball
$B_{2r}(x)\subset\Omega$.
For the proof, we need to introduce a new
function $\bar{u}_1$, which satisfies
\begin{equation*} \left\{
\begin{aligned}
 &\Delta^2 \bar{u}_1=|u|^p,  &~~in  &~~B_{r_0},\\
 & \bar{u}_1=\Delta \bar{u}_1=0,  ~~ ~~~~  &~~on  & ~~\partial B_{r_0}(0),
                          \end{aligned} \right.
\end{equation*}
By the maximum principle, $\bar{u}_1\geq|u_1|\geq0$. By the same
method for $u_1$, we have
\[\int_{B_{r_0}}|u|^p\bar{u}_1\leq C\varepsilon^{1+2\delta}.\]
We can use this to control $|u|^p|u_1|$.
\end{rmk}

\begin{lem}\label{lem ep-regularity}
There exists a universal constant $\varepsilon^\ast>0$ and
$\theta\in(0,1)$, such that if $u$ is a stable solution of
\eqref{equation} satisfying \eqref{estimate of 2p by p+1 sign ch}, and
\[(2R)^{4\frac{p+1}{p-1}-n}\int_{B_{2R}(x_0)}(\Delta u)^2+|u|^{p+1}=\varepsilon\leq\varepsilon_0,\]
then $u$ is smooth in $B_R$, and there exists a universal constant
$C(\varepsilon_0)$ such that
\[\sup_{B_R(x_0)}|u|\leq C(\varepsilon^\ast)R^{-\frac{4}{p-1}}.\]
\end{lem}
\begin{proof}
By choosing a smaller $\varepsilon^\ast$, we can apply Lemma
\ref{lem improvement of decay} to any ball $B_r(x)$ with $x\in
B_R(x_0)$ and $r\leq R/4$, which says
\[(\theta r)^{4\frac{p+1}{p-1}-n}\int_{B_{\theta r}(x)}(\Delta u)^2+|u|^{p+1}\leq
\frac{1}{2}r^{4\frac{p+1}{p-1}-n}\int_{B_{r}(x)}(\Delta
u)^2+|u|^{p+1}\]
iterating the above implies
\[\int_{B_r(x)}(\Delta u)^2+|u|^{p+1}\leq Cr^{n-4\frac{p+1}{p-1}+\delta}\]
for any $x\in B_1$ and $r\leq1/8$. Here $\delta>0$ is a constant
depending only on $\varepsilon_0$ and $\theta$ in Lemma \ref{lem
improvement of decay}. In other words, $u$ belongs to the
homogeneous Morrey space $L^{p+1,n+\delta-4\frac{p+1}{p-1}}(B_1)$.
Then the Morrey space estimate for biharmonic operator gives the
claim, see the appendix.
\end{proof}
Together with a covering argument, this lemma gives a bound on the
Hausdorff dimension of the singular set of $u(=u_{\lambda^\ast}+1)$
\[dim \, {\mathcal S}\leq n-4\frac{p+1}{p-1}.\]
In particular, $u$ is smooth on an open dense set.

For any $x_0\in\Omega$ and $\lambda\in(0,1)$, define the blowing up
sequence
\[u^\lambda(x)=\lambda^{\frac{4}{p-1}}u(x_0+\lambda x),\ \ \ \ \lambda\to0,\]
which is also a stable solution of \eqref{equation} in the ball
$B_{1/\lambda}(0)$.

By rescaling \eqref{estimate of 2p}, for all $\lambda\in(0,1)$ and
balls $B_r(x)\subset B_{1/\lambda}$,
\[\int_{B_r(x)} (u^\lambda)^{2p}\leq Cr^{n-\frac{8p}{p-1}}.\]
By elliptic estimates, $u^\lambda$ is uniformly bounded in
$W^{4,2}_{loc}(\mathbb{R}^n)$. Hence, up to a subsequence of
$\lambda\to 0$, we can assume that $u^\lambda\to u^0$ in
$W^{3,2}_{loc}(\mathbb{R}^n)$ and $L^{p+1}_{loc}(\mathbb{R}^n)$ (by
the same proof of \eqref{convergence in q<p+1}). By testing the
equation for $u^\lambda$ (or the stability condition for
$u^\lambda$) with smooth functions having compact support, and then
taking the limit $\lambda\to+\infty$, we see $u^0$ is a stable
solution of \eqref{equation} in $\mathbb{R}^n$.

We have
\begin{lem}\label{lem blowing up}
 For any $r>0$,
$E(r;0,u^0)=\lim\limits_{r\to 0}E(r;0,u)$. So $u^0$ is homogeneous.
\end{lem}
\begin{proof}
A direct rescaling shows $E(r;0,u^\lambda)=E(\lambda r;x_0,u)$. By
the monotonicity of $E(r;x_0,u)$, we only need to show that, for
every $r>0$,
\[E(r;0,u^0)=\lim\limits_{\lambda\to0}E(r;0,u^\lambda).\]
Because $u^\lambda$ is uniformly bounded in $W^{4,2}(B_r)$ and
$L^{2p}(B_r)$, by the compactness results in Sobolev embedding
theorems and trace theorems, and interpolation between $L^q$ spaces
(see \eqref{convergence in q<p+1}), we have
\[\lim\limits_{\lambda\to+\infty}\int_{B_r}(\Delta u^\lambda)^2=\int_{B_r}(\Delta u^0)^2.\]
\[\lim\limits_{\lambda\to+\infty}\int_{B_r} (u^\lambda)^{p+1}=\int_{B_r} (u^0)^{p+1}.\]
\[u^\lambda\to u^0~~\text{in}~~W^{2,2}(\partial B_r).\]
The last claim implies those boundary terms in $E(r;0,u^\lambda)$
converges to the corresponding ones in $E(r;0,u^0)$. Putting these
together we get the convergence of $E(r;0,u^\lambda)$.

Since for any $r>0$, $E(r;0,u^0)=\mbox{const.}$, by Corollary
\ref{coro characterization of homogeneous solutions}, $u^0$ is
homogeneous.
\end{proof}
Here we note that since $u$ satisfies \eqref{estimate of p+1} for
any ball $B_R(x)\subset\Omega$, so by the same argument as in the
proof of Lemma \ref{lem upper bound on E}, we can prove that
$E(r;x,u)$ is uniformly bounded for all $x$ and $r\in(0,1)$. Since
$E(r;x,u)$ is non-decreasing in $r$, we can define the density
function
\[\Theta(x,u):=\lim\limits_{r\to 0}E(r;x,u).\]
\begin{lem}
%\label{lem density function}
\begin{enumerate}
\item $\Theta(x,u)$ is upper semi-continuous in $x$;
\item for all $x$, $\Theta(x,u)\geq 0$;
\item $x$ is a regular point of $u$ if and only $\Theta(x,u)=0$;
\item there exists a universal constant $\varepsilon_0>0$, $x\in
S(u)$ if and only if $\Theta(x,u)\geq\varepsilon_0$.
\end{enumerate}
\end{lem}
\begin{proof}
By the $W^{4,2}$ regularity of $u$, for any $r>0$ fixed, $E(r;x,u)$
is continuous in $x$. $\Theta(x,u)$ is the decreasing limit of these
continuous functions, thus is upper semi-continuous in $x$.

If $u$ is smooth in a neighborhood of $x$, direct calculation shows
$\Theta(x,u)=0$. Since regular points form a dense set, the upper
semi-continuity of $\Theta$ gives $\Theta\geq 0$.

By Lemma \ref{lem ep-regularity}, if $x$ is a singular point, for
any $r>0$,
\[\int_{B_r(x)}(\Delta u)^2+u^{p+1}\geq\varepsilon_0 r^{n-4\frac{p+1}{p-1}}.\]
In other words, for any $\lambda>0$, for the blowing up sequence
$u^\lambda$ at $x_0$,
\[\int_{B_1(0)}(\Delta
u^\lambda)^2+(u^\lambda)^{p+1}\geq\varepsilon_0.\] Then because
$u^\lambda\to u^0$ in $W^{2,2}_{loc}(\mathbb{R}^n)\cap
L^{p+1}_{loc}(\mathbb{R}^n)$ (see the proof of Lemma \ref{lem
blowing up}),
\begin{eqnarray*}
\int_{B_1(0)}(\Delta u^0)^2+(u^0)^{p+1}&= &\lim\limits_{\lambda\to
0}\int_{B_1(0)}(\Delta
u^\lambda)^2+(u^\lambda)^{p+1}\\
&=&\lim\limits_{\lambda\to 0}
\lambda^{-n+4\frac{p+1}{p-1}}\int_{B_\lambda(0)}(\Delta
u)^2+(u)^{p+1}\geq\varepsilon_0.
\end{eqnarray*}
Hence $u^0$ is nontrivial, and by Remark \ref{rmk E for homogeneous
solutions} and Lemma \ref{lem blowing up},
\[\Theta(x,u)=E(1;0,u^0)\geq c(n,p)\varepsilon_0.\]
Here $c(n,p)$ is a constant depending only on $p$ and $n$.

 On the other hand, if
$\Theta(x,u)<c(n,p)\varepsilon_0$, then by Remark
\ref{rmk E for homogeneous solutions}, for any blow up limit $u^0$ at $x$,
\[\int_{B_1(0)}(\Delta u^0)^2+(u^0)^{p+1}<\varepsilon_0.\]
 Then by the convergence of $u^\lambda$ in
$W^{2,2}_{loc}(\mathbb{R}^n)\cap L^{p+1}_{loc}(\mathbb{R}^n)$, for
$\lambda$ sufficiently small,
\[\lambda^{4\frac{p+1}{p-1}-n}\int_{B_\lambda(x)}(\Delta u)^2+u^{p+1}=\int_{B_1(0)}(\Delta u^\lambda)^2+(u^\lambda)^{p+1}\leq\varepsilon_0.\]
By Lemma \ref{lem ep-regularity}, $u$ is smooth in
$B_{\lambda/2}(x)$. Consequently, $\Theta(x,u)=0$. These finish the
proof of the last two claims.
\end{proof}
\begin{rmk}
If $\lim\limits_{\lambda\to 0}u^\lambda=u^0$ in some sense (for
example, as in the above blowing up sequence) so that for any $x$
and $r>0$, $\lim\limits_{\lambda\to 0}E(r;x,u^\lambda)=E(r;x,u^0)$,
then
\[\lim\limits_{\lambda\to 0}\Theta(x,u^\lambda)\leq\Theta(x;u^0).\]
That is, $\Theta(x;u)$ is also upper semi-continuous in $u$.
\end{rmk}
\begin{lem}
Let $u\in W^{2,2}_{loc}(\mathbb{R}^n)\cap
L^{p+1}_{loc}(\mathbb{R}^n)$ be a homogeneous stable solution of
\eqref{equation} on $\mathbb{R}^n$, satisfying the monotonicity
formula and the integral estimate \eqref{estimate of 2p}, then for
any $x\neq 0$, $\Theta(x,u)\leq\Theta(0,u)$. Moreover, if
$\Theta(x,u)=\Theta(0,u)$, $u$ is translation invariant in the
direction $x$, i.e. for all $t\in\mathbb{R}$,
\[u(tx+\cdot)=u(\cdot)~~a.e. ~~\text{in}~~\mathbb{R}^n.\]
\end{lem}
\begin{proof}
With the help of the integral estimate \eqref{estimate of 2p},
similar to Lemma \ref{lem upper bound on E}, for any
$x_0\in\mathbb{R}^n$,
\[\lim\limits_{r\to+\infty}E(r;x_0,u)\leq C.\]
And we can define the blowing down sequence with respect to the base
point $x_0$,
\[u^\lambda(x)=\lambda^{\frac{4}{p-1}}u(x_0+\lambda x)\ \ \ \ \lambda\to+\infty.\]
Since $u$ is homogeneous with respect to $0$,
\[u^\lambda(x)=u(\lambda^{-1}x_0+x),\]
which converges to $u(x)$ as $\lambda\to+\infty$ in
$W^{2,2}_{loc}(\mathbb{R}^n)\cap L^{p+1}_{loc}(\mathbb{R}^{n})$.
Then Lemma \ref{lem blowing up} can be applied to deduce that
\begin{eqnarray*}
\Theta(0;u)=E(1;0,u)&=&\lim\limits_{\lambda\to+\infty}E(1;0,u^\lambda)\\
&=&\lim\limits_{\lambda\to+\infty}E(\lambda;x_0,u)\\
&\geq&\Theta(x_0;u).
\end{eqnarray*}
Moreover, if $\Theta(x_0;u)=\Theta(0,u)$, the above inequality
become an equality:
\[\lim\limits_{\lambda\to+\infty}E(\lambda;x_0,u)=\Theta(x_0;u).\]
This then implies that $E(\lambda;x_0,u)\equiv\Theta(x_0;u)$ for all
$\lambda>0$. By Corollary \ref{coro characterization of homogeneous
solutions}, $u$ is homogeneous with respect to $x_0$. Then for all
$\lambda>0$,
\[u(x_0+x)=\lambda^{\frac{4}{p-1}}u(x_0+\lambda x)=u(\lambda^{-1} x_0+x).\]
By letting $\lambda\to+\infty$ and noting that
$u(\lambda^{-1}x_0+\cdot)$ are uniformly bounded in
$W^{2,2}_{loc}(\mathbb{R}^n)$, we see
\[u(x_0+\cdot)=u(\cdot)~~\text{a.e. on}~~\mathbb{R}^n.\]

Because $u$ is homogeneous with respect to $0$, a direct scaling
shows that $\Theta(tx_0;u)=\Theta(x_0;u)$ for all $t>0$, so the
above equality still holds if we replace $x_0$ by $tx_0$ for any
$t>0$. A change of variable shows this also holds if $t<0$.
\end{proof}
With this lemma in hand we can apply the Federer's dimension
reduction principle (cf. Appendix A in \cite{S}) to deduce Theorem
\ref{thm partial regularity}.

\appendix
\section{Proof of Estimate in Lemma \ref{lem ep-regularity} }
\setcounter{equation}{0}

\newcommand{\aaa}{\epsilon}

Let us use the notation
$$
\|f\|_{q,\gamma,\Omega}
= \sup_{x,r}
\left(
r^{-\gamma} \int_{B(x,r)\cap \Omega} |f|^q
\right)^{1/q}
$$
$$
L^{q,\gamma}(\Omega) = \{ u \in L^q(\Omega): \|u\|_{q,\gamma,\Omega} <\infty \} ,
$$
where $\Omega \subset \R^n$ is a bounded domain, $0<\gamma\leq n$, $1\leq q <\infty$.

For completeness we give a proof of the following result, which is an adaptation of \cite{pacard-a-note,pacard-convergence}.
\begin{lem}
\label{lemma a1}
Assume $u$ is a weak solution of
$$
\Delta^2 u = |u|^{p-1} u \quad\text{in } B_1(0)
$$
and $u \in L^{p,n-4\frac{p}{p-1}+\delta}(B_1(0))$ for some $\delta>0$. Then $u$ is bounded in $B_{1/2}(0)$.
\end{lem}

We need some preliminaries.
Let
$$
I_\alpha(f)(x)
=
\int_{\R^n} |x-y|^{-n+\alpha} f(y) \, dy
$$

\begin{lem}(\cite[Lemma~1]{pacard-a-note})
If $f\in L^{1,\gamma}(\R^n)$, $0<\aaa<\gamma$ and $1<p<\frac{n-\aaa}{n-\aaa-\alpha}$ then
\begin{align}
\label{riesz morrey space}
\int_\Omega
|I_\alpha(f)|^p(x)\, d x
\leq C diam(\Omega)^{n-\aaa-(n-\alpha-\aaa)p}
\int_\Omega |f| dx
\end{align}
\end{lem}

\begin{lem} (Campanato \cite{campanato})
\label{lemma campanato}
Let $0<\gamma<n$ and $c>0$. Assume $\phi :(0,R]\to\R$ is a nonnegative nondecreasing function such that
$$
\phi(\rho) \leq c \left(
\frac{\rho^n}{r^n} \phi(r) + r^\gamma\right)
\quad \text{for all } 0<\rho\leq r\leq R .
$$
Then there is $C$ depending only on $n,\gamma,c $ such that
$$
\phi(\rho) \leq C \rho^\gamma ( \frac{\phi(r)}{r^\gamma} +1)
\quad\text{for all } 0<\rho\leq r\leq R.
$$
\end{lem}

\begin{lem}
Let $v$ satisfy $\Delta^2v = 0$ in $B_R(0)$. Then there is $C$ such that
\begin{align}
\label{est biharmonic}
|v(x)| \leq \frac{C}{R^n} \int_{B_R(0)} |v| dy
\quad \text{for all } |x|\leq \frac12 R.
\end{align}
\end{lem}
\begin{proof}
By scaling we can restrict to $R=1$ and $v\in C^4(\overline B_1(0))$. Let $\eta\in C^\infty(\R^n)$ be a cut-off function with $\eta(x)=1$ of $|x|\geq \frac14$ and $\eta(x)=0$ for $|x|\leq \frac18$.
Let $\Gamma(x) = c_n |x|^{n-4}$ be the fundamental solution of $\Delta^2$ in $\R^n$, $c_n>0$.
Then
$$
v(x) = \int_{B_1} v(y) \Delta^2 ( \Gamma(x-y) \eta(x-y) ) \, dy
\quad \text{for } |x|\leq \frac 12
$$
and \eqref{est biharmonic} follows.
%Indeed, integrating by parts and using $\Delta^2 v=0$
%\begin{align*}
%\int_{B_1} v(y) \Delta^2 ( \Gamma(x-y) \eta(x-y) ) \, dy
%= \int_{\partial B_1}
%v D_\nu \Delta\Gamma
%- D_\nu v \Delta \Gamma
%+ \Delta v D_\nu \Gamma
%- D_\nu \Delta v \Gamma .
%%\frac{\partial \Delta \Gamma}{\partial \nu}
%\end{align*}
%Then it is standard that the right hand side equals $v(x)$.
\end{proof}

\begin{proof}[Proof of Lemma~\ref{lemma a1}]
Let $R_1<1$ (close to 1), $|x|<R_1$ and $0<r<\frac{1-R_1}{2}$.
Let $u_1 = \Gamma \ast ( |u|^{p-1} u \chi_{B_r(x)} )$ where $\Gamma(x) = c_n |x|^{n-4}$ is the fundamental solution of $\Delta^2$ in $\R^n$, $c_n>0$, and $\chi_{B_r(x)}$ is the indicator function of $B_r(x)$.
Let $u_2 = u-u_1$. Then $\Delta^2 u_2 = 0$ in $B_r(x)$.
By \eqref{est biharmonic}
$$
|u_2(z)| \leq \frac{C}{r^n} \int_{B_r(x)} |u_2|
\quad\text{for } z \in B_{r/2}(x).
$$
Let $y\in B_{r/2}(x)$ and $0<\rho<\frac r2$. Integrating in $B_\rho(y)$ and using
H\"older's inequality
$$
\int_{B_\rho(y)} |u_2|^p
\leq C (\frac{\rho}{r})^n
\int_{B_r(x)} |u_2|^p .
$$
Therefore
\begin{align}
\nonumber
\int_{B_\rho(y)} |u|^p
%& \leq
%C
%\int_{B_\rho(y)} |u_1|^p
%+C\int_{B_\rho(y)} |u_2|^p
%\\
&
\leq
C\int_{B_\rho(y)} |u_1|^p
+
 C (\frac{\rho}{r})^n
\int_{B_r(x)} |u_2|^p
\\
%&
%\leq
%\int_{B_\rho(y)} |u_1|^p
%+
% C (\frac{\rho}{r})^n
%\int_{B_r(x)} |u|^p
%+
% C (\frac{\rho}{r})^n
%\int_{B_r(x)} |u_1|^p
%\\
&
\label{a3}
\leq
 C (\frac{\rho}{r})^n
\int_{B_r(x)} |u|^p
+
C
\int_{B_r(x)} |u_1|^p  .
\end{align}
Let $\gamma_0 = n-4\frac{p}{p-1}+\delta$.
Using \eqref{riesz morrey space} with $\alpha = 4$,
$\gamma = \gamma_0$ and $\aaa$ a number such that
$ n-4\frac{p}{p-1} < \aaa < \gamma_0$
we have
$$
\int_{B_r(x)} |u_1|^p \leq C r^{n-\aaa-(n-4-\aaa)p}
\int_{B_r(x)} |u|^p .
$$
Then, combining with \eqref{a3} we obtain
\begin{align*}
\int_{B_\rho(y)} |u|^p
& \leq
 C (\frac{\rho}{r})^n
\int_{B_r(x)} |u|^p
+ C r^{n-\aaa-(n-4-\aaa)p}
\int_{B_r(x)} |u|^p
\\
&
\leq
 C (\frac{\rho}{r})^n
\int_{B_r(x)} |u|^p
+ C r^{n-\aaa-(n-4-\aaa)p + \gamma_0}
\end{align*}
for any $y\in B_{r/2}(x)$, $0<\rho<\frac r2$.
We have the validity of the inequality for $0<\rho\leq r$, possibly increasing $C$.
Using the Lemma of Campanato (Lemma~\ref{lemma campanato}),
$$
\int_{B_\rho(y)} |u|^p \leq C \rho^{n-\aaa-(n-4-\aaa)p + \gamma_0}
$$
for $0<\rho\leq r$,
which shows that $u \in L^{p,\gamma_1}(B_{R_1})$ where $R_1<1$ can be chosen arbitrarily close to 1, and $\gamma_1 = n-\aaa-(n-4-\aaa)p + \gamma_0$ can be chosen arbitrarily close to $n(1-p)+4 p + \gamma_0$.
In particular we can choose $\gamma_1>\gamma_0$. Repeating the process, we can find a decreasing sequence $R_i \to \frac45$ and an increasing sequence $\gamma_i \to n-4$ such that $u \in L^{p,\gamma_i}(B_{R_i})$. Then by Lemma~\ref{lemma a1} $u \in L^q(B_{3/4}(0)$ for all $q>1$ and by standard elliptic regularity $u\in L^\infty(B_{1/2})$.
\end{proof}


\begin{thebibliography}{50}
\small

\bibitem{campanato}
S. Campanato,
 Equazioni ellittiche del IIdeg ordine espazi {$\mathcal L^{(2,λ)}$}.
{\em Ann. Mat. Pura Appl.} (4) 69 1965 321--381.

\bibitem{CWY}
A. Chang, L. Wang, P. Yang,
   A regularity theory of biharmonic maps,
   {\em Comm. Pure Appl. Math.}
   (9) 52 1999 1113--1137

\bibitem{CEG}
C. Cowan, P. Esposito and N. Ghoussoub, Regularity of extremal
solutions in fourth order nonlinear eigevalue problems on general
domains, {\em DCDS-A} 28 (2010), 1033-1050.



\bibitem{CG}
C. Cowan and N. Ghoussoub, Regularity of semi-stable solutions to
fourth order nonlinear eigevalue problems on general domains,  to appear {\em Cal. Var. PDE} DOI 10.1007/s00526-012-0582-4.


\bibitem{DFG}
J. D\'avila, I. Flores and I. Guerra, Multiplicity of solutions for
a fourth order equation with power-type nonlinearity, {\em Math.
Ann.} 348 (2010), 143-193.

\bibitem{D}
Dancer, E. N. Moving plane methods for systems on half spaces. {\em Math.
Ann.} 342 (2008), no. 2, 245-254.

\bibitem{DDGM}
J. D\'avila, L. Dupaigne, I. Guerra and M. Montenegro, Stable
solutions for the bilaplacian with exponential nonlinearity, {\em
SIAM J. Math. Anal.} 39 (2007), 565-592.

\bibitem{DDF}
J. D\'avila, L. Dupaigne and A. Farina, Partial regularity of finite
Morse index solutions to the Lane-Emden equation, {\em J. Funct.
Anal.} 261 (2011), 218-232.

\bibitem{dupaigne}
L. Dupaigne, Variations elliptiques. Habilitation \`a Diriger des Recherches. 11 december 2011.

\bibitem{DFS}
L. Dupaigne, A. Farina and B. Sirakov, Regularity of the extremal solution, {\em Geometric Partial Differential Equations} , edited by Sc. Norm. Super. Pisa. To appear.

\bibitem{DGGW} L. Dupaigne, M. Ghergu, O. Goubet and G. Warnault,
The Gelfand problem for the biharmonic operator, {\em Arch. Rat. Mech.} DOI:10.1007/s00205-013-0613-028. To appear.


\bibitem{F}
Fleming, Wendell H.,
On the oriented Plateau problem.
{\em Rend. Circ. Mat. Palermo} (2) 11 1962 69-90.

\bibitem{Fa}
A. Farina, On the classification of solutions of the Lane-Emden
equation on unbounded domains of $\R^N$, {\em J. Math. Pures Appl.}
87 (2007), 537-561.

\bibitem{GG}
F. Gazzola and H. -Ch. Grunau, Radial entire solutions for
supercritical biharmonic equations, {\em Math. Ann.} 334 (2006),
905-936.

\bibitem{GNW} C. Gui, W. M. Ni and X. F. Wang, On the stability and instability of positive steady states of a semilinear heat equation in $\mathbb{R}^N$. {\em Comm. Pure Appl. Math.} {\bf 45} (1992), 1153-1181.

\bibitem{GLW}
Yuxia Guo, B. Li and J. Wei, Large energy entire solutions for the
Yamabe type problem of polyharmonic operator. {\em  J. Diff. Eqns.}
254(2013), no.1, 199-228.

\bibitem{GW2} Z.M. Guo and J. Wei, Qualitative properties of entire  radial solutions for a biharmonic equation
with supcritical nonlinearity, {\em Proc. American Math. Soc. } 138
(2010), no.11, 3957-3964.


\bibitem{G-G-S}
F. Gazzola, H.-C. Grunau, G. Sweers,
Polyharmonic boundary value problems. Positivity preserving and
nonlinear higher order elliptic equations in bounded domains.
Lecture Notes in Mathematics, 1991. Springer-Verlag, Berlin, 2010.

\bibitem{HHY}
H. Hajlaoui, A. A. Harrabi and D. Ye, On stable solutions of
biharmonic problem with polynomial growth, {\em arXiv:1211.2223v2}
(2012).
\bibitem{Ka}
P. Karageorgis, Stability and intersection properties of solutions
to the nonlinear biharmonic equation, {\em Nonlinearity} 22 (2009),
1653-1661.
\bibitem{L-L-N}
Y. Li, C.-S. Lin, L. Nirenberg, Nonexistence results to
cooperative systems with supercritical expoenents in
$\mathbb{R}^n_+$, {\em preprint}.

\bibitem{pacard-a-note}
F. Pacard, F.
 A note on the regularity of weak solutions of {$−\Delta u=u^\alpha$} in {$\R^n$}, $n\geq3$.
{\em Houston J. Math.} 18 (1992), no. 4, 621--632.

\bibitem{pacard-convergence}
F. Pacard, Convergence and partial regularity for weak solutions of some nonlinear elliptic equation: the supercritical case.
{\em Ann. Inst. H. Poincaré Anal. Non Linéaire } 11 (1994), no. 5, 537--551.

\bibitem{pacard-mm}
F. Pacard, Partial regularity for weak solutions of a nonlinear elliptic equation.
{\em Manuscripta Math. } 79 (1993), no. 2, 161--172.

\bibitem{P-Q-S}
P. Pol\'{a}cik, P. Quittner, P. Souplet,
Singularity and decay estimates in superlinear problems via Liouville-type
theorems. I. Elliptic equations and systems. {\em Duke Math. J.} 139
(2007), no. 3, 555-579.
\bibitem{P-S 1}
P. Pucci and J. Serrin, A general variational identity, {\em Indiana
Univ. Math. J.} 35(1986), 681-703.
\bibitem{P-S 2}
 P. Pucci and J. Serrin, Critical
exponents and critical dimensions for polyharmonic opera- tors, {\em J.
Math. Pures Appl.} 69(1990), 55-83.

\bibitem{Re}
F. Rellich, Perturbation theory of eigenvalue problems, Gordon and
Breach Science Publisher, New York, (1969).

\bibitem{Ro}
G.V. Rozenblum, The distribution of the discrete spectrum for
singular differential operators, {\em Dokl. Akad. SSSR} 202 (1972),
1012-1015.

\bibitem{S}
L. Simon, Lectures on geometric measure theory. Proceedings of
the Centre for Mathematical Analysis, Australian National
University, 3. Australian National University, Centre for
Mathematical Analysis, Canberra, 1983.

\bibitem{Si} B. Sirakov, Existence results and a priori bounds for higher order elliptic equations and systems, {\em J. Math. Pures Appl.} 89 (2008), 114-133.


\bibitem{Soup} P. Souplet, The proof of the Lane-Emden conjecture in four space dimensions, {\em Adv.
  Math.} 221 (2009), 1409-1427.

\bibitem{WX}
J. Wei and X. Xu,  Classification of solutions of high order conformally invariant equations, {\em Math. Ann.} 313(2) (1999), 207-228.


\bibitem{W}
K. Wang, Partial regularity of stable solutions to the
supercritical equations and its applications, {\em Nonlinear Anal.} 75
(2012), no. 13, 5238-5260.

\bibitem{Y-W}
D. Ye, J. Wei, Liouville Theorems for finite Morse index
solutions of Biharmonic problem, {\em Math. Ann.} to appear.

\end{thebibliography}
\end{document}